\documentclass{preprint}
\usepackage[utf8x]{inputenc}
\usepackage[english]{babel}
\usepackage{xargs}
\usepackage{stmaryrd}

\usepackage{fmtcount}
\usepackage{amsmath,amssymb,amsthm,enumitem,xcolor,tikz, graphicx, mathtools}
\usepackage[osf,upint,libertine]{newtx}
\usepackage{microtype}

\usepackage{mhequ}
\usepackage{anyfontsize}
\usepackage{aligned-overset}
\usepackage{csquotes}
\usepackage[backend = biber,
style = trad-plain,
giveninits = false,
isbn = false,
doi = true,
url = false,
eprint = true,
maxbibnames=99,
]{biblatex}
\usepackage{mleftright}
\mleftright
\usepackage{zref-clever}

\usepackage[unicode=true,pdfusetitle,
 bookmarks=true,bookmarksnumbered=false,bookmarksopen=false,
 breaklinks=true,pdfborder={0 0 0},pdfborderstyle={},backref=false,colorlinks=true]
 {hyperref}

\makeatletter
\DeclareFontFamily{OMX}{MnSymbolE}{}
\DeclareSymbolFont{MnLargeSymbols}{OMX}{MnSymbolE}{m}{n}
\SetSymbolFont{MnLargeSymbols}{bold}{OMX}{MnSymbolE}{b}{n}
\DeclareFontShape{OMX}{MnSymbolE}{m}{n}{
    <-6>  MnSymbolE5
   <6-7>  MnSymbolE6
   <7-8>  MnSymbolE7
   <8-9>  MnSymbolE8
   <9-10> MnSymbolE9
  <10-12> MnSymbolE10
  <12->   MnSymbolE12
}{}
\DeclareFontShape{OMX}{MnSymbolE}{b}{n}{
    <-6>  MnSymbolE-Bold5
   <6-7>  MnSymbolE-Bold6
   <7-8>  MnSymbolE-Bold7
   <8-9>  MnSymbolE-Bold8
   <9-10> MnSymbolE-Bold9
  <10-12> MnSymbolE-Bold10
  <12->   MnSymbolE-Bold12
}{}

\let\llangle\@undefined
\let\rrangle\@undefined
\DeclareMathDelimiter{\llangle}{\mathopen}%
                     {MnLargeSymbols}{'164}{MnLargeSymbols}{'164}
\DeclareMathDelimiter{\rrangle}{\mathclose}%
                     {MnLargeSymbols}{'171}{MnLargeSymbols}{'171}
\makeatother

\hypersetup{allcolors=blue}
\hypersetup{final}

\numberwithin{equation}{section}
\numberwithin{figure}{section}
\numberwithin{table}{section}
\AddToHook{env/myproposition/begin}{%
\zcsetup{countertype={mytheorem=proposition}}}
\theoremstyle{plain}
\newtheorem{thm}{\protect\theoremname}[section]
    \AddToHook{env/thm/begin}{\zcsetup{countertype={thm=theorem}}}
\theoremstyle{definition}

    \AddToHook{env/defn/begin}{\zcsetup{countertype={thm=definition}}}

    \AddToHook{env/example/begin}{\zcsetup{countertype={thm=example}}}
\newtheorem*{example*}{\protect\examplename}

\theoremstyle{remark}
\newtheorem{rem}[thm]{\protect\remarkname}
    \AddToHook{env/rem/begin}{\zcsetup{countertype={thm=remark}}}
\theoremstyle{plain}
\newtheorem{prop}[thm]{\protect\propositionname}
    \AddToHook{env/prop/begin}{\zcsetup{countertype={thm=proposition}}}
\theoremstyle{plain}

    \AddToHook{env/conjecture/begin}{\zcsetup{countertype={thm=conjecture}}}
\theoremstyle{plain}
\newtheorem{lem}[thm]{\protect\lemmaname}
    \AddToHook{env/lem/begin}{\zcsetup{countertype={thm=lemma}}}
\newlist{thmstepnv}{enumerate}{4}
\setlist[thmstepnv]{leftmargin=*,align=left,wide,labelwidth=!,itemindent=!,labelindent=0pt}
\setlist[thmstepnv,1]{label={\itshape {\thmstepname} \arabic*.},ref=\arabic*}
\setlist[thmstepnv,2]{label={\itshape {\thmstepname} {\thethmstepnvi\alph*}.},ref=\thethmstepnvi\alph*}
\setlist[thmstepnv,3]{label={\itshape {\thmstepname\ \alph*.}},ref=\alph*}
\setlist[thmstepnv,4]{label={\itshape {\thmstepname} \arabic*.},ref=\arabic*}
\theoremstyle{plain}

\newlist{casenv}{enumerate}{4}
\setlist[casenv]{leftmargin=*,align=left,widest={iiii}}
\setlist[casenv,1]{label={{\itshape\ \casename} \arabic*.},ref=\arabic*}
\setlist[casenv,2]{label={{\itshape\ \casename} \roman*.},ref=\roman*}
\setlist[casenv,3]{label={{\itshape\ \casename\ \alph*.}},ref=\alph*}
\setlist[casenv,4]{label={{\itshape\ \casename} \arabic*.},ref=\arabic*}

\zcsetup{
    abbrev=true, %
    cap=true,
    nameinlink=false
    }

\zcRefTypeSetup{equation}{
Name-sg = ,
name-sg = ,
Name-pl = ,
name-pl = ,
refbounds = {(,,,)},
+refbounds-rb = {(,,,},
+refbounds-re = {,,,)},
rangesep = {--}
}

\newcommand{\cref}[1]{\zcref{#1}}
\newcommand{\Cref}[1]{\zcref[S]{#1}}

\newcommand{\crefrange}[2]{\zcref[range,rangetopair=false,endrange=stripprefix]{#1,#2}}
\renewcommand{\eqref}{\cref}

\DeclareFontFamily{U}{matha}{\hyphenchar\font45}
\DeclareFontShape{U}{matha}{m}{n}{
      <5> <6> <7> <8> <9> <10> gen * matha
      <10.95> matha10 <12> <14.4> <17.28> <20.74> <24.88> matha12
      }{}
\DeclareSymbolFont{matha}{U}{matha}{m}{n}
\DeclareFontSubstitution{U}{matha}{m}{n}

\DeclareFontFamily{U}{mathx}{\hyphenchar\font45}
\DeclareFontShape{U}{mathx}{m}{n}{
      <5> <6> <7> <8> <9> <10>
      <10.95> <12> <14.4> <17.28> <20.74> <24.88>
      mathx10
      }{}
\DeclareSymbolFont{mathx}{U}{mathx}{m}{n}
\DeclareFontSubstitution{U}{mathx}{m}{n}

\DeclareMathDelimiter{\vvvert}{0}{matha}{"7E}{mathx}{"17}

\providecommand{\casename}{Case}
\providecommand{\conjecturename}{Conjecture}
\providecommand{\corollaryname}{Corollary}
\providecommand{\definitionname}{Definition}
\providecommand{\examplename}{Example}
\providecommand{\assumptionname}{Assumption}
\providecommand{\lemmaname}{Lemma}
\providecommand{\propositionname}{Proposition}
\providecommand{\remarkname}{Remark}
\providecommand{\theoremname}{Theorem}
\providecommand{\thmstepname}{Step}

\global\long\def\e{\mathrm{e}}%
\global\long\def\Cov{\mathrm{Cov}}%

\global\long\def\RR{\mathbb{R}}%
\global\long\def\EE{\mathbb{E}}%
\global\long\def\NN{\mathbb{N}}%

\global\long\def\clG{\mathcal{G}}%
\global\long\def\clF{\mathcal{F}}%
\global\long\def\clW{\mathcal{W}}%
\global\long\def\clC{\mathcal{C}}%
\global\long\def\clZ{\mathcal{Z}}%
\global\long\def\clQ{\mathcal{Q}}%
\global\long\def\clA{\mathcal{A}}%
\global\long\def\clU{\mathcal{U}}%
\global\long\def\fkm{\mathfrak{m}}%
\global\long\def\fkn{\mathfrak{n}}%
\global\long\def\scrX{\mathscr{X}}%

\global\long\def\Frob{\mathrm{F}}%
\global\long\def\Id{\mathrm{Id}}%
\global\long\def\Law{\operatorname{Law}}%
\global\long\def\tr{\operatorname{tr}}%
\global\long\def\Var{\operatorname{Var}}%
\global\long\def\Cov{\operatorname{Cov}}%
\global\long\def\Lip{\operatorname{Lip}}%
\global\long\def\dif{\mathrm{d}}%
\global\long\def\Dif{\mathrm{D}}%

\def\Z{\mathcal Z}
\def\f{\frac}
\def\B{\mathcal B}

\renewbibmacro{in:}{}
\AtEveryBibitem{\clearfield{month}}
\AtEveryBibitem{\clearfield{primaryclass}}
\AtEveryBibitem{\clearfield{eprintclass}}
\DeclareSourcemap{
  \maps[datatype=bibtex]{
    \map{
      \pertype{book}
      \step[fieldsource=pages, final]
      \step[fieldset=pagetotal, origfieldval, final]
      \step[fieldset=pages, null]
    }
  }
}
\DefineBibliographyStrings{english}{volume = {vol\adddot}, volumes = {vol\adddot}}

\begin{document}

\title{A critical stochastic heat equation\\with long-range noise}

\author{Alexander Dunlap,\texorpdfstring{$^1$}{} Martin Hairer,\texorpdfstring{$^{2,3}$}{} and Xue-Mei Li\texorpdfstring{$^{2,3}$}{}}
\def\email#1{\renewcommand{\ttdefault}{cmtt}{email: \tt#1}}
\institute{Duke University, USA, \email{alexander.dunlap@duke.edu}
\and EPFL, Switzerland, \email{\{martin.hairer,xue-mei.li\}@epfl.ch}
\and Imperial College London, UK
}

\maketitle

\begin{abstract}
    We consider a semilinear stochastic heat equation in spatial dimension at least $3$, forced by a noise that is white in time with a covariance kernel that decays like $\lvert x\rvert^{-2}$ as $\lvert x\rvert\to\infty$. We show that in an appropriate diffusive scaling limit with a logarithmic attenuation of the noise, the pointwise statistics of the solution can be approximated by the solution to a forward-backward stochastic differential equation (FBSDE). The scaling and structure of the problem is similar to that of the two-dimensional stochastic heat equation forced by an approximation of space-time white noise considered by the first author and Gu (\emph{Ann. Probab.}, 2022). However the resulting FBSDE is different due to the long-range correlations of the noise.
\end{abstract}
\setcounter{tocdepth}{2}
\tableofcontents

\section{Introduction}

In this paper, we are interested in semilinear stochastic heat equations roughly of the form
\begin{equation}\label{eq:formal-eqn}
  \dif u_t = \frac12 \Delta u_t + \sigma(u_t)\dif W_t,
\end{equation}
where $(\dif W_t)$ is a white-in-time noise with spatial covariance given by the Riesz kernel $\lvert\cdot\rvert^{-2}$.
We note two important features of this stochastic heat equation. First, the correlation function is not compactly supported, and indeed has a rather heavy non-integrable tail, so we consider the noise to have 
``long-range correlations'' like those considered by Gerolla and the second two authors in \cite{GHL23}.
Second, this equation is \emph{critical}, in that if for some $\lambda>0$ we set $u^\lambda(t,x) = u(\lambda^2 t,\lambda x)$, then $u^\lambda$ formally satisfies the same equation \cref{eq:formal-eqn}. Stochastic heat equations exhibiting this type of formal scale-invariance of have been of significant recent interest in the literature, which we will review below.

Even in the additive and multiplicative cases $\sigma(u) = \alpha\in\RR$ or $\sigma(u)=\beta u$ for some $\beta\in\RR$,
the equation \cref{eq:formal-eqn} does not yield function-valued solutions in spatial dimension $d\ge 2$. Since we do not expect any additional regularity from taking a more general $\sigma$, it is not clear how to interpret
the nonlinearity $\sigma(u)$ for general $\sigma$, nor (if this quantity can be interpreted) the product with the irregular noise. See \cite{MT04}, discussed further below, for \emph{measure}-valued
solutions in the multiplicative case.
One way to address this difficulty is to regularize the covariance at some small scale $\rho>0$, and then try to take a limit $\rho\to 0$. %
In this work, we show that if in addition to mollifying the nonlinearity in this way we also attenuate by a logarithmic factor in $\rho$, then we can obtain interesting limits of the pointwise values of the solution as $\rho \downarrow 0$.

Let us now introduce our setting precisely. Let $\fkm,\fkn,d\in\NN$ with $d\ge 3$. We study the $\RR^\fkm$-valued stochastic heat equation
\begin{equ}
    \dif u_t (x) = \frac12\Delta u_t(x)\dif t + \frac{\sigma(u_t(x))}{\sqrt{\log \rho^{-1}}}\dif W^\rho_t(x),\qquad t\in\RR,x\in\RR^d.\label{eq:ueqn}
\end{equ}
Here, $\Delta$ denotes the spatial Laplacian and $\sigma\colon\RR^\fkm\to\RR^{\fkm}\otimes\RR^\fkn$ is a uniformly Lipschitz nonlinearity, where the space of $\fkm\times\fkn$ matrices $\RR^{\fkm}\otimes\RR^\fkn$ is equipped with the Frobenius norm  $|\sigma|_F=\sqrt{\sigma \sigma^T}$ (the $L^2$ norm of the vector of singular values). %
The $\RR^\fkn$-valued noise $(\dif W^\rho_t(x))$ is adapted to a temporal filtration $\{\clF_t\}_t$ and has correlation structure
\begin{equation}
    \EE[\dif W^\rho _t(x)\dif W^\rho_{t'}(x')^\top]=\delta(t-t')R^\rho(x-x')\Id_\fkn,
    \label{eq:noise-covariance}
\end{equation}
where we define
\begin{equation}
    R(x) = \lvert x\rvert^{-2}\qquad\text{and}\qquad R^\rho = G_{2\rho}*R.\label{eq:Rdef}
\end{equation}
Here, $G_t(x) = (2\pi t)^{-d/2}\e^{-\lvert x\rvert^2/(2t)}$ is the standard $d$-dimensional heat kernel and $*$ denotes spatial convolution. In the sequel, we will use the notation $\clG_tv=G_t *v$ for convolution with the heat kernel. We consider solutions to \cref{eq:ueqn} in the mild sense, namely
\begin{equation}
    u_t^\rho(x) = \clG_t u_0 (x) + \frac1{\sqrt{\log \rho^{-1}}}\int_0^t \clG_{t-s}[\sigma(u^\rho_s)\,\dif W^\rho _s](x).\label{eq:umild}
\end{equation}

Our aim is to study pointwise limits of \cref{eq:umild}, as $\rho\downarrow 0$, in this setting of both \emph{scaling criticality} and \emph{long-range dependence}. Our main result is a close analogue of that in \cite{DG22,DG23a}, where the noise covariance 
is taken to be a scale-$\rho^{1/2}$ approximation of a delta function (in particular the correlation function has compact support). Namely, we show that the law of the solution at
any given location asymptotically coincides with the composition of the heat flow 
and the law of a suitable forward-backward stochastic differential equation (FBSDE). The FBSDE arising in our case is \begin{subequations}
  \label{eq:FBSDE}
\begin{align}
    \dif \Gamma_{a,Q}(q) &= J(Q-q,\Gamma_{a,Q}(q))\dif B(q),\qquad q\in (0,Q);\label{eq:FBSDE-SDE}\\
    \Gamma_{a,Q}(0) &= a;\label{eq:FBSDE-ic}\\
    J(r,b) &= \frac{\EE[\sigma(\Gamma_{b,r}(r))]}{\sqrt{2(d-2)}},\qquad r\in (0,Q),\label{eq:FBSDE-J}
\end{align}
\end{subequations}
where $(B(q))_{q\in [0,Q]}$ is a standard $\RR^\fkn$-valued Brownian motion.
We will show in \cref{prop:Jcharacterization} below that, if $\Lip(\sigma)<\sqrt{2(d-2)}$, then there is a
unique $J\colon [0,1]\to \RR^\fkm\otimes\RR^\fkm$ such that the resulting solutions to \crefrange{eq:FBSDE-SDE}{eq:FBSDE-ic} satisfy \cref{eq:FBSDE-J}. In other words, the FBSDE \cref{eq:FBSDE} has a unique solution for $Q\in [0,1]$. We note that the diffusivity in \cref{eq:FBSDE-SDE} can alternatively be written as
\begin{equation}\label{eq:rewrite-diffusivity}
J(Q-q,\Gamma_{a,Q}(q)) = \left.\frac{\EE[\sigma(\Gamma_{b,Q-q}(Q-q))]}{\sqrt{2(d-2)}}\right\rvert_{b=\Gamma_{a,Q}(q)}
=\f{\EE \left[\sigma(\Gamma_{a,Q}(Q))\mid \Gamma_{a,Q}(q)\right]}{\sqrt{2(d-2)}}.
\end{equation}
This is because the process $(\Gamma_{a,Q}(q+r))_{r\in [0,Q-q]}$
solves the SDE problem
\[
  \Gamma_{a,Q}(q+r) = \Gamma_{a,Q}(q) + \int_0^r J(Q-q-s,\Gamma_{a,Q}(q+s))\,\dif B(q+s),\qquad r\in [0,Q-q],
\]
which, up to a time translation of $B$,  is the same integral equation as that solved by $(\Gamma_{b,Q-q}(r))_{r\in [0,Q-q]}$ if we condition on $\Gamma_{a,Q}(q) = b$ (in the sense of stochastic flows).
Therefore, taking $r=Q-q$, we obtain $\Law(\Gamma_{a,Q}(Q)\mid \Gamma_{a,Q}(q) = b) = \Law(\Gamma_{b,Q-q}(Q-q))$, and %
\cref{eq:rewrite-diffusivity} follows. The function $J$ is called the \emph{decoupling function} of the FBSDE. We discuss it in more detail in \cref{subsec:FBSDE} below.

We denote by $\clU_{s,t}^\rho$ the (nonlinear) propagator of \cref{eq:ueqn}, 
so $\clU_{s,t}^\rho u_s$ denotes the solution at time $t$ of \cref{eq:SPDE-rescaled} with initial data $u_{s}\colon \RR^d\to \RR^\fkm$ at time $s$.  Our main theorem states:

\begin{thm}\label{thm:mainthm}
Suppose that $\Lip(\sigma)<\sqrt {d-2}$.
 Then, for any (possibly random) initial condition $u_0$, independent of the noise $W$, such that $$\sup_{x\in\RR^d} \EE \lvert u_0(x)\rvert ^\ell<\infty$$ 
 for some $\ell>2$, the following holds for any fixed $t>0$ and $x\in\RR^d$, 
  \begin{equ}\label{eq:Wasserstein-convergence}
    \lim_{\rho\downarrow 0}\clW^2\bigl((\clU_{0,t}^\rho u_0)(x),\Gamma_{  \clG_tu_0(x),1}(1) \bigr)=0,
    \end{equ}
where $\clW^2$ denotes the Wasserstein-$2$ distance.
\end{thm}

\begin{rem}
 While the restriction $\beta<\sqrt{d-2}$ is used in the proof, it is unlikely that this is optimal. In fact, \cref{thm:nosecondmomentphasetrans} below strongly suggests that there is no critical
 value of $\beta$ in the present case of noise with long-range correlations. Similarly, we expect that the restriction $\beta<\sqrt{2(d-2)}$ for the well-posedness of \eqref{eq:FBSDE-SDE} is suboptimal as well.
\end{rem}

\begin{rem}
    If $U$ solves the SPDE
    \begin{equation}\label{eq:SPDE-rescaled}
        \dif U_t(x) = \frac12\Delta U_t(x)\dif t + \frac{\sigma(U_t(x))}{\sqrt{\log\rho^{-1}}}\dif W^1_t(x),
    \end{equation}
 then the law of its diffusively scaled solution $ (U_{\rho^{-1}t}(\rho^{-1/2}x))_{t,x}$ is the same as the law of $(u_t(x))_{t,x}$, as can be seen by  the scaling relations $\delta(t\rho)= \f 1 \rho \delta(t)$ and $R^\rho_{2\rho}(\f x {\sqrt{\rho}})=G_2*R$. Thus, the limit $\rho\downarrow 0$ of the SPDE \cref{eq:ueqn} corresponds to the long-time, large-space behavior of an SPDE with fixed noise covariance. Note however 
 that, in order to obtain pointwise fluctuations of order $1$, it is still required to attenuate the noise strength as time is taken to infinity, hence the continued presence of the division by $\sqrt{\log\rho^{-1}}$ in~\cref{eq:SPDE-rescaled}.
\end{rem}

At a high level, the reason for the appearance of the FBSDE is the same as in \cite{DG22,DG23a}, and is discussed in more detail in their respective introductions. Roughly speaking, the logarithmic attenuation of the noise strength means the quadratic variation of the stochastic integral appearing in \cref{eq:umild} has a type of mean-field behavior and thus can be approximated by the solution to a one-dimensional equation. The self-similar structure of the equation allows us to write this one-dimensional equation as an FBSDE.

The key novelty in the present work however, is that the FBSDE \cite[(1.5--1.7)]{DG22} is qualitatively different. 
In particular, \cite[(1.7)]{DG22} reads \begin{equation}\label{eq:J-compactlysupported}J(q,a) = c_1(\EE [\sigma^2(\Gamma_{a,q}(q))])^{1/2},\end{equation} rather than $J(q,a) = c_2\EE [\sigma(\Gamma_{a,q}(q))]$ as in \cref{eq:FBSDE-J}. The constants $c_1,c_2\in(0,\infty)$ are a matter of normalization and irrelevant for the present discussion.

The reason for the difference in $J$ comes from the covariance structure of the noise. The quadratic variation of the stochastic integral in \cref{eq:umild} is 
\begin{equation}\label{eq:QV-intro}
  \frac1{\log\rho^{-1}}\int_0^t\iint G_{t-s}(x-y_1)G_{t-s}(x-y_2)R^\rho(y_1-y_2)\sigma(u_s(y_1))\sigma^\top(u_s(y_2))\,\dif y_1\,\dif y_2\,\dif s.
\end{equation}
In the case of compactly supported covariance kernel considered in \cite{DG22,DG23a}, the quantity $\sigma(u_s(y_1))\sigma^\top(u_s(y_2))$ can be approximated by $(\sigma\sigma^\top)(s,u_s((y_1+y_2)/2))$, since the compactly supported covariance kernel $R^\rho$ forces $y_1\approx y_2$. This leads to the appearance of $\sigma^2$ in \cref{eq:J-compactlysupported}. On the other hand, in the present setting when $R^\rho$ is long-range, the spatial decorrelation of $u_s$ implies that the spatial integral in \cref{eq:QV-intro} more closely resembles $\EE[\sigma(u_s)]\EE[\sigma(u_s)]^\top$, which leads to the form of \cref{eq:FBSDE-J}. In this discussion, we have glossed over the fact that the quadratic variation in fact has to be estimated ``at every scale,'' and the expectations here really need to be conditioned on the spatial average of $u_s$ at the present scale, which leads to the conditioning on the right side of \cref{eq:rewrite-diffusivity}.

This phenomenon is very similar to that observed in comparing the results of \cite{GL19} with those of \cite{GHL23} on Edwards--Wilkinson fluctuations of stochastic heat equations in settings with supercritical scaling. In \cite{GL19}, compactly supported noise is considered, and the effective noise strength of the resulting Edwards--Wilkinson field involves the microscopic two-point correlation function of $\sigma$ applied to the stationary solution. (See \cite[(1.3)]{GL19}.) In \cite{GHL23}, long-range noise is considered, and the effective noise strength is simply the expectation of $\sigma$ applied to the stationary solution and spatial correlation survives in the limit. (See \cite[Theorem~1.3]{GHL23}.)

\subsection{Properties of the forward-backward SDE\label{subsec:FBSDE}}
In order to make full use of \cref{thm:mainthm}, one of course needs to understand the solutions to the FBSDE \cref{eq:FBSDE}. We refer to \cite{MY99} for an introduction to the theory of FBSDEs, and \cite{DG22,DG23a,DG23b} for more detailed discussion of an FBSDE quite similar (but, as noted above, not identical) to \cref{eq:FBSDE}.

A standard application of Itô's formula (see e.g.\ \cite[§8.2]{MY99} or \cite[§3]{DG23a}) shows that, 
at least formally speaking, the decoupling function $J$ satisfies the quasilinear PDE
\begin{subequations}\label{eq:JPDEproblem}
    \begin{align}
      \partial_q J_{k,\ell}(q,b) &= \frac12 \tr [(JJ^\top)\nabla_b^2J_{k,\ell}](q,b),\quad q>0,b\in\RR^\fkm,1\le k\le \fkm,1\le \ell\le\fkn; \label{eq:JPDE}\\
        J(0,b) &= \frac{\sigma(b)}{\sqrt{2(d-2)}},\qquad b\in\RR^\fkm.\label{eq:Jic}
    \end{align}
\end{subequations}
Here we use the notation %
$\nabla_b^2$ for the Hessian. %
The reason we say this derivation of the PDE problem \cref{eq:JPDEproblem} is only formal is that the equation \cref{eq:JPDE} may be degenerate parabolic, as we have not made any assumption to guarantee that $JJ^\top$ has full rank. %
 The interpretation of the problem \cref{eq:JPDEproblem} can be made precise as in \cite[§3]{DG23a}. We do not go into detail in the present work since we will not need this interpretation here, %
 although it is implicit in proof of \cref{prop:explicit-sqrtsum} below. The PDE formulation has been useful in \cite{DG23b} in studying the FBSDE derived in \cite{DG22,DG23a}. We expect that a similar study of \cref{eq:JPDEproblem} could be fruitful in obtaining a deeper understanding of the decoupling function in the present setting.

Of course, for most choices of $\sigma$, we do not expect to find an explicit solution to the FBDSE \cref{eq:FBSDE}. However, we know of the following family of $\sigma$s where the solution is explicit: %
    \begin{prop}\label{prop:explicit-sqrtsum}
        Suppose that $\fkm =\fkn= 1$, $\alpha,\beta\ge 0$ and $\sigma(b) = \sqrt{\alpha + \beta b^2}$. The corresponding solution to the FBSDE \cref{eq:FBSDE} satisfies \begin{equation}J(q,b) = \sqrt{\frac{\alpha \e^{\beta q/(4(d-2))}+\beta b^2}{2(d-2)}}.\label{eq:Jexplicit}\end{equation}
    \end{prop}
    See \cref{subsec:explicit} for a proof of \cref{prop:explicit-sqrtsum}, which essentially consists of checking that this choice of $J$ solves \cref{eq:JPDEproblem} and justifying the PDE in this setting.
    Of particular interest are the cases when one of $\alpha$ and $\beta$ is $0$. When $\alpha \ge 0$ and $\beta = 0$, we are considering the case of additive noise. In this case, the SPDE \cref{eq:ueqn} is simple to analyze, as the solution is (conditional on the initial data) a Gaussian process. This is consistent with the fact that, according to \cref{eq:Jexplicit}, the decoupling function $J$ is constant in this case, so the solutions to the FBSDE will also be Gaussian.

    \subsection{The multiplicative-noise case and the existence of a phase transition}
    The case $\alpha = 0$ and $\beta\ge 0$ in \cref{eq:Jexplicit} corresponds to the multiplicative noise case $\sigma(u) = \beta u$, which is of particular interest due to its connection both with directed polymers and with the KPZ equation via the Cole--Hopf transform \cite{BG97,GHL24}. The multiplicative problem with compactly supported noise covariance function in $d=2$ has been studied extensively. Of particular relevance to the present work is \cite{CSZ17}, in which it was shown that, for $\beta<\sqrt{2\pi}$, the limiting pointwise distribution is log-normal. From the point of view of the FBSDE, this is because the solution to the corresponding FBDSE is a deterministic time-change of a geometric Brownian motion; see \cite[§1.3]{DG22}. The log-normal random variable blows up as $\beta\uparrow\sqrt{2\pi}$, pointing to the phase transition at $\beta=\sqrt{2\pi}$. At $\beta=\sqrt{2\pi}$ (in fact in a window of size $O(1/\sqrt{\log\rho^{-1}})$ around this critical value), a measure-valued process called the \emph{critical stochastic heat flow} has been constructed in \cite{CSZ21} and characterized axiomatically in \cite{Tsa24}; see also \cite{BC98,CSZ19,GQT19}.

    In our present setting, when $\alpha = 0$ we see from \cref{eq:Jexplicit} that $J(q,b) = b\sqrt {\beta/(2(d-2))}$, 
    so that the solution to the FBSDE \crefrange{eq:FBSDE-SDE}{eq:FBSDE-J} is a geometric Brownian motion. However, unlike in the setting considered in \cite{CSZ17}, it remains finite for all $\beta<\infty$. This suggests that the requirement imposed in \cite{CSZ17} that $\beta$ be smaller than a critical value may be unnecessary in the present setting. As further evidence for this, we prove the following theorem.
\begin{thm}\label{thm:nosecondmomentphasetrans}
    Let $\fkm = \fkn=1$,  $\beta \in (0,\infty)$, $\sigma(u) = \beta u$, and $(u_t(x))$ solve \cref{eq:ueqn} with initial condition $u_0 \equiv 1$. Then, for any $T\in (0,\infty)$, we have a constant $C=C(\beta,T)<\infty$, independent of $\rho$, such that
    \[
        \sup_{t\in [0,T],x\in\RR^d}\EE |u_t(x)|^2\le C.
    \]
\end{thm}
In other words, unlike for the problem studied in \cite{BC98,CSZ17}, there is no value of $\beta$ at which the pointwise second moment of the solution blows up. Thus, if there is to be a phase transition, the behavior must be significantly more subtle.

A similar problem to \cref{eq:ueqn} with $\fkm=\fkn=1$ and $\sigma(u)=\beta u$, but \emph{without} the division by $\sqrt{\log\rho^{-1}}$, was studied by Mueller and Tribe in \cite{MT04}. (See \cite{KT24} for some limiting behaviors of this process.) In that case, the problem with $\rho = 0$ can be studied directly, and a measure-valued solution is obtained. We can alternatively think of this problem as \cref{eq:ueqn} with $\sigma(u)=u\sqrt{\log\rho^{-1}}$. Thus, if we believe the conjecture that the analogue of the critical value $\beta=\sqrt{2\pi}$ from \cite{CSZ21} is $\beta=\infty$ in this setting, then we could see the process constructed in \cite{MT04} as an analogue of the critical stochastic heat flow.

A discrete version of the multiplicative problem with similar noise correlation decay as $\lvert x\rvert\to \infty$ but in spatial dimension $d=2$ has recently been considered in \cite{CCD25}. In that setting, similar log-normal behavior is again observed, but in that case the power of $\log\rho^{-1}$ in the attenuation is different, and there is a phase transition in $\beta$. The study of the semilinear case in that setting is an interesting problem for future work.

\subsection{Proof strategy and organization of the paper}
The large-scale structure of the proof of \cref{thm:mainthm} is similar to the structure of the proof in \cite{DG22}. The key tool that allows us to approximate the solution to the (infinite-dimensional) SPDE by the solution to the (finite-dimensional) FBSDE is the technique of turning off the noise on approximate time intervals, which are ``long'' from the point of view of the diffusive scaling of the heat equation but ``short'' from the point of view of the effect of the noise on the solution. This means that turning off the noise on these intervals does not affect the solution much but allows the solution to relax enough to become approximately constant on appropriate mesoscopic scales, which allows us to approximate the problem by a finite-dimensional Markov chain. The relevant key estimates are proved in \cref{sec:turn-off-noise} and the Markov chain is defined and studied in \cref{sec:analysis-markov-chain}.

The remaining step, and the source of key difference between our setting and that of \cite{DG22,DG23a}, is to determine the form of the FBSDE relation \cref{eq:FBSDE-J}. To study this, we must compute the variance of the increments of the approximating Markov chain. The key estimate in this direction is \cref{prop:VarofStep} in \cref{sec:statistics}, which relies on an estimate of the covariance of a function of the solution established in \cref{sec:cov-bd}. This covariance estimate resembles \cite[Lemma~2.13]{GHL23} and is proved using based on Malliavin calculus techniques.

We note that previous works on semilinear stochastic heat equations in the critical dimension at subcritical temperature have proved a wider range of results than we show here, e.g.\ multipoint statistics \cite{CSZ17,DG22,DG23a} and Edwards--Wilkinson behavior for the rescaled and recentered random field \cite{CSZ17,Tao22,DG24}. The logarithm of the solution in the linear case has also been studied, and Edwards--Wilkinson behavior for the corresponding field has been proved \cite{CSZ20,Gu20} (see also \cite{GHL24} for a version of this result in the long-range supercritical regime), as well as some mesoscopic averaging results \cite{Tao23}. We believe that analogues of these results should hold in the present setting as well. We have chosen to keep the present paper as simple as possible in order to focus on one of the key novelties in this setting, namely the different form of the FBSDE. We leave the extensions in these various directions to future work (perhaps once the question on the necessity of the Lipschitz bound on $\sigma$ can be addressed as well).

The paper is organized as follows.
In \cref{sec:FBSDE}, a well-posedness theory for the FBSDE \cref{eq:FBSDE} is established. In \cref{sec:integral-estimates}, we establish some results about various integrals that are used in the results throughout the paper.
In \cref{sec:moment-bd}, we prove an upper bound on the moments of the solution, adapting the proof of the moment bound in \cite{DG23a}. In \cref{sec:turn-off-noise}, we implement the ``turning off the noise'' strategy described above. In \cref{sec:cov-bd}, we prove the necessary bounds on covariances of functions of the solution. In \cref{sec:approx-decoupling}, we study the approximate decoupling function, and in particular prove the key estimate \cref{prop:VarofStep}. In \cref{sec:analysis-markov-chain}, we define and study the approximating Markov chain and prove \cref{thm:mainthm}. Finally, in \cref{sec:analysis-second-moment}, we prove \cref{thm:nosecondmomentphasetrans}.

\subsection{Notation}\label{sec:notation}
We let $\scrX^\ell_t$ denote the space of $\clF_t$-measurable random fields $v$ such that the norm
\begin{equ}[e:triplenorm]
    \vvvert v\vvvert_\ell \coloneqq\sup_{x\in\RR^d}(\EE\lvert v(x)\rvert^\ell)^{1/\ell}
\end{equ}
is finite. We will mostly work with the $2$-norm, so we abbreviate $\vvvert\cdot\vvvert \coloneqq \vvvert\cdot\vvvert_2$ and $\scrX_t \coloneqq \scrX^2_t$.
For matrices, we use the notation $\lvert\cdot\rvert_\Frob$ for the Frobenius norm. Unless otherwise stated, we assume that $\RR^\fkm\otimes \RR^\fkn$ is equipped with the Frobenius norm.
We also use the Japanese bracket notation $\langle u\rangle = \sqrt{1+\lvert u\rvert^2}$ for $u\in\RR^k$ (for any $k\in\NN$).

\subsection{Acknowledgments}
A.~D.\ was supported by the National Science Foundation under grants DMS-2002118 and DMS-2346915.
X.-M.~L. acknowledges support by Swiss National Science Foundation project mint 10000849 and  UKRI  grant
EP/V026100/1.

\section{The forward-backward SDE\label{sec:FBSDE}}
In this section we establish some basic results about the forward-backward SDE. The notation and arguments in this section follow \cite{DG23a}.

Suppose that $g\colon [0,Q]\times\RR^\fkm\to \RR^\fkn\otimes\RR^\fkn$ is continuous and also that it is
Lipschitz in its second argument, uniformly in its first argument. For each $a\in\RR^\fkm$, we let $(\Theta^g_{a,Q})_{q\in [0,Q]}$ solve the SDE
\begin{subequations}\label{eq:Thetaproblem}
\begin{align}
    \dif \Theta^g_{a,Q}(q) &= g(Q-q,\Theta^g_{a,Q}(q))\dif B(q),\qquad q\in (0,Q);\label{eq:ThetaSDE}\\
    \Theta^g_{a,Q}(0) &= a.\label{eq:ThetaIC}
\end{align}
\end{subequations}
For a Lipschitz function $f\colon\RR^\fkm\to \RR^\fkm\otimes\RR^\fkn$, we define
\begin{equation}\clQ_fg(Q,a) = \EE[f(\Theta^g_{a,Q}(Q))],\label{eq:Qdef}\end{equation}
and we note that \cref{eq:FBSDE-J} is equivalent to the condition
\begin{equation}\label{eq:Qfp}\clQ_{\sigma/\sqrt{2(d-2)}}J = J.\end{equation}
We define the function space
\[
    \clA_Q \coloneqq \Bigl\{J\in\clC([0,Q]\times\RR^\fkm;\RR^\fkm\otimes \RR^\fkn)\ :\ \sup_{q\in [0,Q]} \Lip(J(q,\cdot))<\infty\Bigr\}.
\]
\begin{prop}\label{prop:Jcharacterization}Let $f\in\Lip(\RR^\fkm;\RR^\fkm\otimes\RR^\fkn)$. For any $Q\in [0,\Lip(f)^{-2})$, there is a unique $J\in\clA_Q$ such that $\clQ_fJ=J$.\end{prop}
    \begin{proof}
        By Jensen's inequality and the convexity of $\lvert\cdot\rvert^2_\Frob$, the following analogue of \cite[(2.19)]{DG23a} holds for all $Q\in[0,\Lip(f)^{-2})$ and all $a\in\RR^\fkm$:
        \begin{equation}
            \lvert\clQ g(Q,a)\rvert^2_\Frob = \lvert\EE[\sigma(\Theta^g_{a,Q}(Q))]\rvert^2_\Frob \le \EE\lvert \sigma(\Theta^g_{a,Q}(Q))\rvert^2_\Frob.\label{eq:Qg}
        \end{equation}
        The proof of \cite[Proposition~2.5]{DG23a} now works essentially verbatim, with \cref{eq:Qg} replacing \cite[(2.25)]{DG23a} and the similarly obvious fact that
        \begin{equ}
            \lvert\EE[\sigma_1]-\EE[\sigma_2]\rvert^2_\Frob \le \EE\lvert\sigma_1-\sigma_2\rvert^2_\Frob
        \end{equ}
        (for any matrix-valued random variables $\sigma_1,\sigma_2$) replacing \cite[Proposition~A.2]{DG23a}.
    \end{proof}
    \cref{prop:Jcharacterization} and standard well-posedness results for SDEs immediately imply that if $\Lip(\sigma)<\sqrt{2(d-2)}$, then the FBSDE \cref{eq:FBSDE} has a unique solution.

    \subsection{Explicit solutions\label{subsec:explicit}}
    In this section we give a proof of \cref{prop:explicit-sqrtsum} regarding a 
    particular family of explicit solutions to the FBSDE.
    \begin{proof}[Proof of \cref{prop:explicit-sqrtsum}]
      Morally, the proof consists of checking that the solution \cref{eq:Jexplicit} satisfies the quasilinear heat equation \cref{eq:JPDEproblem}. However, since we have not established that PDE problem in this setting, we will work via Itô's formula directly.

        Let $f(q) = \alpha \e^{\beta q/(4(d-2))}$. With this definition and the choice \cref{eq:Jexplicit} of $J$, the SDE \crefrange{eq:FBSDE-SDE}{eq:FBSDE-ic} becomes
        \begin{align*}
            \dif \Gamma_{a,Q}(q) &= \sqrt{\frac{f(Q-q)+\beta \Gamma_{a,Q}(q)^2}{2(d-2)}}\dif B(q),\qquad q\in (0,Q);\\
            \Gamma_{a,Q}(0) &= a.
        \end{align*}
        Let \begin{equation}N_q(Q,a)\coloneqq \EE[J(Q-q,\Gamma_{a,Q}(q))].\label{eq:Ndef}\end{equation}
        We want to apply Itô's formula to $q\mapsto N_q(Q,a)$. This is not a problem if $\alpha>0$, since in that case $J$ is smooth. If $\alpha =0$, then $J$ fails to be smooth at $(0,0)$, but in that case the solution to the SDE will almost surely never reach 0, so we can use an approximation argument analogous to \cite[Section~3]{DG23a}.
        Thus, using that
        \[\partial_q J(q,b) = \frac{f'(q)}{\sqrt{2(d-2)(f(q)+\beta b^2)}}\]
        and
        \[\partial_b^2 J(q,b) = \frac{\beta f(q)}{\sqrt{2(d-2)}(f(q)+\beta b^2)^{3/2}},\]
        we conclude that
        \begin{equation}\label{eq:Nnochange}
            \partial_q N_q(Q,a) = \EE\left[\frac{-f'(Q-q)+\beta f(Q-q)/(4(d-2))}{\sqrt{2(d-2)(f(Q-q)+\beta \Gamma_{a,Q}(q)^2}}\right] =0.%
                            \end{equation}
        This means that
        \[\frac{\EE[\sigma(\Gamma_{a,Q}(Q))]}{\sqrt{2(d-2)}} \overset{\cref{eq:Jexplicit}}=  \EE[J(0,\Gamma_{a,Q}(Q))] \overset{\cref{eq:Ndef}}=  N_Q(Q,a) \overset{\cref{eq:Nnochange}}= N_0(Q,a) \overset{\cref{eq:Ndef}}= J(Q,a),
            \]
            which shows that \cref{eq:FBSDE-J} is satisfied, thus completing the proof.
    \end{proof}

    \section{Integral estimates\label{sec:integral-estimates}}
    In this section we collect some estimates and identities for various integrals that we will use in the sequel.
The proofs of these results are somewhat simpler when the Gaussian kernel is used as the mollifier in \cref{eq:Rdef}, so we stick to this choice for expositional clarity. Indeed, the results of the paper can be extended easily to more general choices of mollifier.

    We note that
    \begin{equation}
        \frac1{\lvert x\rvert^2} = \frac12(2\pi)^{d/2}\int_0^\infty \nu^{d/2-2}G_\nu(x)\,\dif \nu,\qquad x\in \RR^d\setminus\{0\},\label{eq:decompose1overx2}
    \end{equation}
    so that 
    \begin{equation}
        R^\rho(x) = G_{2\rho}*\lvert\cdot\rvert^{-2} = \frac12(2\pi)^{d/2}\int_0^\infty \nu^{d/2-2}G_{\nu+2\rho}(x)\,\dif \nu.\label{eq:decomposeRrho}
    \end{equation}
    We will also frequently use the fact that
    \begin{equation}
    (G_t*R^\rho)(0)\overset{\cref{eq:decomposeRrho}}{=}\frac12 \int_0^\infty \frac{\nu^{d/2-2}} {(t+\nu+\rho)^{d/2}}\,\dif\nu = \frac1{(d-2)(t+2\rho)},\label{eq:HKwithkernelat0}
    \end{equation}
    and its immediate consequence that
    \begin{equation}
        \iint R^\rho(y_1-y_2)\prod_{i=1}^2 G_{T-t}(x-y_i)\,\dif y_1\,\dif y_2
        = \frac1{2(d-2)(T-t+\rho)}.\label{eq:applyHKwithkernelat0}
\end{equation}
We will also use the fact that, since $\log\lvert\cdot\rvert^{-1}$ is a superharmonic function on $\RR^d$ (here we use the assumption that $d>2$), %
\begin{equation}
    (G_t*\log\lvert\cdot\rvert^{-1})(x)\le \log \lvert x\rvert^{-1}\qquad\text{for all }x\in\RR^d\setminus\{0\}.\label{eq:Gtlog-1ltlog}
\end{equation}
Moreover, we have an absolute constant $C<\infty$ such that $R^1(x)\le C\lvert x\rvert^{-2}$ for all $x\in\RR^d\setminus\{0\}$, which means that
\begin{equation}
  R^\rho(x)=\rho^{-1}R^1(\rho^{-1/2}x)\le C\rho^{-1}\lvert\rho^{-1/2}x\rvert^{-2}=C\lvert x\rvert^{-2}\qquad\text{for all }\rho>0.\label{eq:Rrhobound}
\end{equation}

We will also need the following estimate.
\begin{prop}\label{prop:integrateHKRrho}
    There is a constant $C$ depending only on $d$ such that
    \begin{equation}\label{eq:integrateHKRrho}
        \int_0^t (G_{2r}*R^\rho)(x)\,\dif r\le C\left(\log\frac{t^{1/2}}{\lvert x\rvert} + 1+\frac{\lvert x\rvert}{t^{1/2}}\right).
    \end{equation}
\end{prop}
We note that the estimate \cref{eq:integrateHKRrho} is quite suboptimal when $\lvert x\rvert \gg t^{1/2}$, but in our applications this will not matter since the decay of the heat kernel will dominate in the superdiffusive regime.
We will prove \cref{prop:integrateHKRrho} as the consequence of a few lemmas.
\begin{lem}We have, for all $t,\rho >0$, that
    \begin{equation}
        \int_0^t (G_s*R^\rho)(x)\,\dif s=\frac1{d/2-1}\left((G_{2\rho}-G_{t+2\rho})*\log\lvert\cdot\rvert^{-1}\right)(x).\label{eq:intGsconvolvedwithRrho}
        \end{equation}
    \end{lem}
    \begin{proof}
        We can compute
        \begin{align}
            F(x)&\coloneqq \frac{d-2}{(2\pi)^{d/2}}\int_0^t (G_s*R^\rho)(x)\,\dif s\label{eq:Fdef}\\
            \overset{\cref{eq:decomposeRrho}}&=(d/2-1)\int_0^t\int_0^\infty \nu^{d/2-2}G_{\nu+s+2\rho}(x)\,\dif \nu \,\dif s\notag\\
                                             &= (d/2-1)\int_0^\infty G_{\nu + 2\rho}(x) \int_0^{t\wedge \nu}(\nu-s)^{d/2-2}\,\dif s\,\dif\nu\notag\\
                                             &= \int_0^\infty G_{\nu+2\rho}(x)\nu^{d/2-1}\,\dif \nu -\int_t^\infty G_{\nu+2\rho}(x)(\nu-t)^{d/2-1}\,\dif \nu\notag\\
                                             &= \int_0^\infty \left(G_{\nu+2\rho}(x)-G_{\nu+t+2\rho}(x)\right)\nu^{d/2-1}\,\dif\nu.\label{eq:rhoasGaussian}
        \end{align}
        Now, from this expression, we can write
        \begin{align*}
            \nabla F(x) &= \int_0^\infty \left(\nabla G_{\nu+2\rho}(x)-\nabla G_{\nu+t+2\rho}(x)\right)\nu^{d/2-1}\,\dif\nu\\
                        &= (G_{2\rho}*Q)(x)-(G_{t+2\rho}*Q)(x),
        \end{align*}
    where we have defined
    \begin{align*}
        Q(x)\coloneqq \int_0^\infty \nabla G_\nu(x)\nu^{d/2-1}\,\dif \nu &= -x\int_0^\infty \nu^{d/2-2}G_\nu(x)\,\dif \nu\\
        \overset{\cref{eq:decompose1overx2}}&= -\frac{2^{1-d/2}x}{\pi^{d/2}\lvert x\rvert^2} = \frac{2^{1-d/2}}{\pi^{d/2}}\nabla(\log\lvert\cdot\rvert^{-1})(x).
    \end{align*}
    Therefore, we have
    \[
        \nabla F(x) = 2^{1-d/2}{\pi^{d/2}}\nabla\left((G_{2\rho}-G_{t+2\rho})*\log\lvert \cdot\rvert^{-1}\right)(x).
    \]
    Now it is clear from the definition \cref{eq:Fdef} that $\lim_{\lvert x\rvert\to\infty} F(x) = 0$, and since we also observe that
    \[
        \lim_{\lvert x\rvert \to\infty}\frac{2^{1-d/2}}{\pi^{d/2}}\left((G_{2\rho}-G_{t+2\rho})*\log\lvert\cdot\rvert^{-1}\right)(x)=0,
    \]
    we see that
    \[
        F(x) = \frac{2^{1-d/2}}{\pi^{d/2}}\left((G_{2\rho}-G_{t+2\rho})*\log\lvert\cdot\rvert^{-1}\right)(x).
    \]
Again recalling \cref{eq:Fdef} we obtain \cref{eq:intGsconvolvedwithRrho}.
\end{proof}

\begin{prop}
    There is a constant $C$ depending only on $d$ such that,
for any $t>0$ and $x\in\RR^d$, we have
\begin{equation}
    (G_t*\log\lvert\cdot\rvert^{-1})(x) \ge -\frac12\log t-C\left(1+\frac{\lvert x\rvert}{t^{1/2}}\right).\label{eq:logconvlb}
\end{equation}
\end{prop}
As with \cref{prop:integrateHKRrho}, this bound is highly suboptimal for $x\gg t^{1/2}$, but again we do not need much precision in that regime.
\begin{proof}
    Define $f(x) \coloneqq (G_t*\log\lvert \cdot\rvert^{-1})(x)$. We can compute
    \begin{equation}
        \begin{aligned}
            f(0) = \int G_t(x)\log\lvert x\rvert^{-1}\,\dif x &= \int G_1 (x)\log \lvert t^{1/2}x\rvert^{-1}\,\dif x\\
                                                         &= -\frac12 \log t + \int G_1(x)\log\lvert x\rvert^{-1}\,\dif x.
        \end{aligned}\label{eq:f0}
    \end{equation}
    We can also compute
    \[\nabla f(y) = -\int G_t(y-x)\frac x{\lvert x\rvert^2}\,\dif x,\]
    so
    \begin{equation}
        \lvert \nabla f(y)\rvert \le \int \frac{G_t(y-x)}{\lvert x\rvert}\,\dif x\le \int \frac{G_t(x)}{\lvert x\rvert}\,\dif x=t^{-1/2}\int \frac{G_1(x)}{\lvert x\rvert}\,\dif x.\label{eq:gradientfest}
    \end{equation}
    Here the second inequality is by \cite{Ibr56}. Combining \cref{eq:f0,eq:gradientfest}, we obtain \cref{eq:logconvlb}.
\end{proof}
\begin{proof}[Proof of \cref{prop:integrateHKRrho}]
We can estimate
\begin{align*}
    \int_0^t& (G_{2s}*R^\rho)(x)\,\dif s = \frac12 \int_0^{2t}(G_s*R^\rho)(x)\,\dif s\\
    \overset{\cref{eq:intGsconvolvedwithRrho}}&= \frac1{d-2}\left((G_{2\rho}-G_{2(t+\rho)})*\log\lvert \cdot\rvert^{-1}\right)(x)\\
                                              &\le \frac1{d-2}\left(\log\lvert x\rvert^{-1}+\frac12\log t + C\left(1+\frac{\lvert x\rvert}{t^{1/2}}\right)\right)\le C\left(\log\frac{t^{1/2}}{\lvert x\rvert}+1+\frac{\lvert x\rvert}{t^{1/2}}\right),
\end{align*}
with the second inequality by \cref{eq:Gtlog-1ltlog,eq:logconvlb},  and where $C$ has changed in the last inequality.
\end{proof}
\section{Moment bound\label{sec:moment-bd}}
In this section, we prove a bound on the moments of $u_t(x)$. The following proposition and its proof are adaptations of \cite[Proposition~3.4]{DG23a}. Recall the notation  $\vvvert \cdot\vvvert_\ell$ from \cref{sec:notation}, as well as the Japanese bracket $\langle\cdot\rangle$.
\begin{prop}\label{prop:momentbd}
    Suppose that $\beta < \sqrt{2(d-2)}$ and $M\in [0,\infty)$ are such that $\lvert\sigma(u)\rvert^2_\Frob \le M + \beta^2\lvert u\rvert^2$ for all $u\in\RR^\fkm$. Suppose
     also that $(u_t(x))$ solves \cref{eq:ueqn}. Then, for all $T>0$, we have a constant $C = C(M,\beta,T)\in (1,\infty)$ and an $\ell_0 = \ell_0(M,\beta,T)>2$ such that, if $\ell \in [2,\ell_0]$, $t\in [0,T]$, and $\rho\in(0,C^{-1})$, then 
    \begin{equation}
        \vvvert u_t\vvvert_\ell \le C\langle \vvvert u_0\vvvert_\ell\rangle.\label{eq:momentbd}
    \end{equation}
\end{prop}
\begin{proof}
    We define the martingale
    \begin{equ}
        U_t = \clG_{T-t}u_t(x)\overset{\cref{eq:umild}}= \clG_Tu_0(x) +\frac1{\sqrt{\log\rho^{-1}}}\int_0^t \clG_{T-s}[\sigma(u_s)\,\dif W^\rho_s](x)\;,
    \end{equ}
    as well as the differential quadratic variation
    \begin{equation}\label{eq:UQVAdef}
        \begin{aligned}
            A_t &\coloneqq \frac{\dif}{\dif t}[U]_t \\&= \frac1{\log\rho^{-1}}\iint R^\rho(y_1 - y_2)G_{T-t}(x-y_1)G_{T-t}(x-y_2)\sigma(u_t(y_1))\sigma(u_t(y_2))^\top\,\dif y_1\,\dif y_2.
    \end{aligned}
    \end{equation}
Here, the notation $[U]_t$ denotes the matrix whose $(i,j)$ entry is the quadratic covariation of the $i$th and $j$th components of $(U_t)_t$ at time $t$. We also define $Y_t = \lvert U_t\rvert^\ell$. By Itô's formula (as in \cite[(4.15)]{DG23a}), we obtain
\[
    \dif Y_t = \ell\lvert U_t\rvert^{\ell-2}U_t\cdot \dif U_t + \frac12 \tr\left[(\ell(\ell-2)\lvert U_t\rvert^{\ell-4}U_t^{\otimes 2}+\ell\lvert U_t\rvert^{\ell-2}\Id_\fkm)\dif [U]_t\right]\;.
\]
Integrating and taking expectations, noting that the local martingale term vanishes since it is indeed a martingale as $U_t$ has all moments bounded uniformly in $t$ in a compact set (for fixed $\rho$), we see that 
\begin{align}
    \EE&\lvert u_T(x)\rvert^\ell = \EE Y_T\notag\\
       &\le \vvvert u_0\vvvert^\ell_\ell +\frac12\EE\left[\int_0^T \tr\left[(\ell(\ell-2)\lvert U_t\rvert^{\ell-4}U_t^{\otimes 2}+\ell\lvert U_t\rvert^{\ell-2}\Id_\fkm)\,\dif [U]_t\right]\right]\notag\\
       & = \vvvert u_0\vvvert^\ell_\ell +\frac12\int_0^T \tr\EE\left[(\ell(\ell-2)\lvert U_t\rvert^{\ell-4}U_t^{\otimes 2}+\ell\lvert U_t\rvert^{\ell-2}\Id_\fkm)A_t\right]\,\dif t.\label{eq:exptY}
\end{align}
Now we estimate
\begin{equation}
    \tr\left[\lvert U_t\rvert^{\ell-4}U_t^{\otimes 2}A_t\right] \le \lvert U_t\rvert^{\ell-2}\tr A_t,
\end{equation}
so \cref{eq:exptY} becomes
\begin{equation}\label{eq:Utbd}
\EE\lvert u_T(x)\rvert^\ell \le \vvvert u_0\vvvert^\ell_\ell +\frac12(\ell-1)\int_0^T\EE \left[\lvert U_t\rvert^{\ell-2}\tr A_t\right]\,\dif t.
\end{equation}
Thus we want to estimate the quantity
\begin{equation}\label{eq:momentbd-step3}
    \begin{aligned}
\EE&\left[\lvert U_t\rvert^{\ell-2}\tr A_t\right]\\
   &\le\frac1{\log\rho^{-1}}\EE\left[\lvert U_t\rvert^{\ell-2}\iint R^\rho(y_1-y_2)\prod_{i=1}^2 (G_{T-t}(x-y_i)\lvert\sigma(u_t(y_i))\rvert_\Frob\,\dif y_1\,\dif y_2\right],
\end{aligned}
\end{equation}
where we used Hölder's inequality, the definition \cref{eq:UQVAdef} of $A_t$, and the submultiplicativity of the Frobenius norm.
The assumption on $\sigma$, along with Hölder's inequality as in \cite[(4.21)]{DG23a}, imply that
\[\lvert \sigma(u)\rvert^\ell_\Frob\le 2^{\ell/2-1}[M^{\ell/2}+\beta^\ell\lvert u\rvert^\ell].\]
Using Young's product inequality and Jensen's inequality, we conclude that for any (deterministic) $Z\in(0,\infty)$,
\begin{align*}
    \EE&\left[\lvert U_t\rvert^{\ell-2}\lvert \sigma(u_t(y_1))\rvert_\Frob\lvert \sigma(u_t(y_2))\rvert_\Frob\right]\\
       &=\EE\left[\left(Z^{\frac{\ell-2}{\ell}}\lvert U_t\rvert^{\ell-2}\right)\left(Z^{-\frac{\ell-2}{2\ell}}\lvert \sigma(u_t(y_2))\rvert_\Frob\right)\left(Z^{-\frac{\ell-2}{2\ell}}\lvert \sigma(u_t(y_2))\rvert_\Frob\right)\right]\\
       &\le (1-2/\ell)Z\EE\lvert U_t\rvert^\ell +\frac1{\ell Z^{\frac{\ell-2}2}}\sum_{i=1}^2 \EE\lvert\sigma(u_t(y_i))\rvert^\ell_\Frob\\
       &\le(1-2/\ell)Z\vvvert u_t\vvvert^\ell_\ell + \frac{2^{\ell/2}}{\ell Z^{\frac{\ell-2}2}}\left(M^{\ell/2}+\beta^\ell\vvvert u_t\vvvert^\ell_\ell\right).
\end{align*}
We minimize the right side by taking
\[
    Z\coloneqq 2^{\frac{\ell-2}\ell}\left(\beta^\ell+\frac{M^{\ell/2}}{\vvvert u_t\vvvert_\ell^\ell}\right)^{2/\ell},
\]
which yields
\begin{align}
\EE & \left[\lvert U_t\rvert^{\ell-2}\lvert\sigma(u_t(y_1))\rvert_{\Frob}\lvert\sigma(u_t(y_2))\rvert_{\Frob}\right]\nonumber \\
 & \le(1-2/\ell)2^{(\ell-2)/2}\left(\beta^{\ell}+\frac{M^{\ell/2}}{\vvvert u_{\ell}\vvvert_{\ell}^{\ell}}\right)\vvvert u_t\vvvert_{\ell}^{\ell}+\frac{2^{\frac{\ell-2}{\ell}}\left(M^{\ell/2}+\beta^{\ell}\vvvert u_t\vvvert_{\ell}^{\ell}\right)}{\ell\left(\beta^{\ell}+\frac{M^{\ell/2}}{\vvvert u_t\vvvert_{\ell}^{\ell}}\right)^{(\ell-2)/\ell}}\nonumber \\
 & =2^{\frac{\ell-2}2}(1-1/\ell)\vvvert u_t\vvvert_{\ell}^{\ell-2}\left(M^{\ell/2}+\beta^{\ell}\vvvert u_t\vvvert_{\ell}^{\ell}\right)^{2/\ell}\nonumber \\
 & \le2^{\frac{\ell-2}2}\left(1-1/\ell\right)\left(M\vvvert u_t\vvvert_{\ell}^{\ell-2}+\beta^{\ell}\vvvert u_t\vvvert_{\ell}^{\ell}\right)\nonumber \\
 & \le2^{1-2/\ell}(1-1/\ell)\left(\frac2{\ell}M^{\ell/2}+\left(\beta^{\ell}+1-2/\ell\right)\vvvert u_t\vvvert_{\ell}^{\ell}\right)\nonumber \\
 & \le\frac2{\ell}M^{\ell/2}+\left(\beta^{\ell}+1-2/\ell\right)\vvvert u_t\vvvert_{\ell}^{\ell}.\label{eq:applyJensenparty}
\end{align}

Using \cref{eq:applyJensenparty} and \cref{eq:applyHKwithkernelat0}
in \cref{eq:momentbd-step3}, and then substituting into \cref{eq:Utbd},
we get
\begin{equation}
\vvvert u_{T}\vvvert_{\ell}^{\ell}\le\vvvert u_{0}\vvvert_{\ell}^{\ell}+\frac{\ell(\ell-1)}{4(d-2)\log\rho^{-1}}\int_{0}^{T}\frac{\frac2{\ell}M^{\ell/2}+(\beta^{\ell}+1-2/\ell)\vvvert u_t\vvvert_{\ell}^{\ell}}{T-t+\rho}\,\dif t.\label{eq:momentbd-step3-1}
\end{equation}
Defining $\Upsilon^{\ell}\coloneqq\sup\limits _{t\in[0,T]}\vvvert u_t\vvvert_{\ell}^{\ell}$,
we get
\begin{equation}
\Upsilon^{\ell}\le\vvvert u_{0}\vvvert_{\ell}^{\ell}+\frac{\ell(\ell-1)}{4(d-2)}\left(2M^{\ell/2}/\ell+(\beta^{\ell}+1-2/\ell)\Upsilon^{\ell}\right)\log_{\rho^{-1}}(1+T\rho^{-1}).\label{eq:upsilonbd}
\end{equation}
Now, since we assumed that $\beta<\sqrt{2(d-2)}$, there is an $\ell_{0}>2$
such that, if $\ell\in[2,\ell_{0}]$, then for sufficiently small
$\rho$ we have
\[
\frac{\ell(\ell-1)}{4(d-2)}(\beta^{\ell}+1-2/\ell)\log_{\rho^{-1}}(1+T\rho^{-1})<1,
\]
and in this case we can rearrange \cref{eq:upsilonbd} to imply
\cref{eq:momentbd} (with an appropriate choice of~$C$).
\end{proof}

\section{Turning off the noise\label{sec:turn-off-noise}}

In this section, we implement the ``turning off the noise'' strategy developed in \cite{DG22,DG23a} for the case of compactly
supported noise covariance. The estimates we obtain are very similar in form to the corresponding estimates of \cite[§4]{DG22}.

First we show that we do not change the solution of the equation much if we ``turn off'' the noise for a sufficiently short period of time, compared to the time at which we evaluate the solution. In the statement of the proposition below, we note that  $\clU_{\tau_2,t}\clG_{\tau_2-\tau_1}u_{\tau_1}$ represents the solution of the equation run until time $t$ but with the noise ``turned off'' on the time interval $[\tau_1,\tau_2]$, whereas $\clU_{\tau_1,t}u_{\tau_1}$ represents the solution of the equation run until time $t$ without turning off the noise.

\begin{prop}
    \label{prop:turnoffnoise}For every $T>0$ and $\lambda<\sqrt{2(d-2)}$, there
is a $C=C(T,\lambda)<\infty$ such that the following holds.
Suppose that $\Lip(\sigma)<\lambda$.
Let $\tau_1\in \RR$ and let $u_{\tau_1}\in\mathscr{X}_{\tau_1}$.
Then we have, for $t>\tau_2>\tau_1$, that %
\[
\vvvert\mathcal{U}_{\tau_2,t}\mathcal{G}_{\tau_2-\tau_1}u_{\tau_1}-\mathcal{U}_{\tau_1,t}u_{\tau_1}\vvvert^2\le C\langle\vvvert u_{\tau_1}\vvvert\rangle^2\frac{\log\left(\frac{t-\tau_1+\rho}{t-\tau_2+\rho}\right)+1}{\log\rho^{-1}}.
\]
\end{prop}
\begin{proof}
Let $u_t=\mathcal{U}_{\tau_1,t}u_{\tau_1}$ and $\tilde{u}_t=\mathcal{U}_{\tau_2,t}\mathcal{G}_{\tau_2-\tau_1}u_{\tau_1}$.
We can write
\[
\tilde{u}_t(x)=\mathcal{G}_{t-\tau_1}u_{\tau_1}(x)+\frac1{\sqrt{\log\rho^{-1}}}\int_{\tau_2}^t\clG_{t-s}[\sigma(\tilde{u}_{s})\,\dif W_{s}^{\rho}](x)
\]
while
\[
u_t(x)=\mathcal{G}_{t-\tau_1}u_{\tau_1}(x)+\frac1{\sqrt{\log\rho^{-1}}}\int_{\tau_1}^t\clG_{t-s}[\sigma(u_t)\,\dif W_{s}^{\rho}](x).
\]
Therefore, we have
\begin{align*}
u_t(x)-\tilde{u}_t(x) & =\frac1{\sqrt{\log\rho^{-1}}}\int_{\tau_1}^{\tau_2}\clG_{t-s}[\sigma(u_{s})\,\dif W_{s}^{\rho}](x)\\
 & \qquad+\frac1{\sqrt{\log\rho^{-1}}}\int_{\tau_2}^t\clG_{t-s}[(\sigma(u_{s})-\sigma(\tilde{u}_{s}))\,\dif W_{s}^{\rho}](x).
\end{align*}
We can estimate the second moment as
\begin{align*}
\mathbb{E} & \lvert u_t(x)-\tilde{u}_t(x)\rvert^2\\
 & \le\frac1{\log\rho^{-1}}\int_{\tau_1}^{\tau_2}\iint R^{\rho}(y_1-y_2)\mathbb{E}\left[\prod_{i=1}^2G_{t-s}(x-y_{i})\lvert\sigma(u_{s}(y_{i}))\rvert_{\Frob}\right]\,\dif y_1\,\dif y_2\,\dif s\\
 & \quad+\frac1{\log\rho^{-1}}\int_{\tau_2}^t\iint R^{\rho}(y_1-y_2)\mathbb{E}\left[\prod_{i=1}^2G_{t-s}(x-y_{i})\left\lvert\sigma(u_{s}(y_{i}))-\sigma(\tilde{u}_{s}(y_{i}))\right\rvert_{\Frob}\right]\\&\qquad\qquad\qquad\qquad\qquad\qquad\qquad\qquad\qquad\qquad\qquad\qquad\qquad\,\dif y_1\,\dif y_2\,\dif s\\
\overset{\cref{eq:applyHKwithkernelat0}}&{\le}\frac1{2(d-2)\log\rho^{-1}}\int_{\tau_1}^{\tau_2}\frac{\vvvert\sigma(u_{s})\vvvert^2}{t-s+\rho}\,\dif s+\frac{\Lip(\sigma)^2}{2(d-2)\log\rho^{-1}}\int_{\tau_2}^t\frac{\vvvert u_{s}-\tilde{u}_{s}\vvvert^2}{t-s+\rho}\,\dif s\\
                                        \overset{\cref{eq:momentbd}}&{\le}C\langle\vvvert u_{\tau_1}\vvvert\rangle^2\frac{\log\frac{t-\tau_1+\rho}{t-\tau_2+\rho}}{\log\rho^{-1}}+\frac{\Lip(\sigma)^2}{2(d-2)\log\rho^{-1}}\int_{\tau_2}^t\frac{\vvvert u_{s}-\tilde{u}_{s}\vvvert^2}{t-s+\rho}\,\dif s.
\end{align*}
Then we conclude by \cite[Lemma 4.3]{DG22} (with the appropriate change of constants).
\end{proof}

In the next proposition, we show that after the noise has been turned off for some period, the solution is quite flat (due to the smoothing influence of the heat kernel), so we do not lose much if we replace the solution by a constant (its value at a fixed point $X\in\RR^d$) at that time. This of course is only true if we evaluate the solution at a point $x$ that is close to $X$ compared to the amount of time elapsed. To state this concisely, for $X\in\RR^{d}$, we let $\mathcal{Z}_{X}$ be the ``freezing'' operator that, given a function $u$ defined on $\RR^d$, returns the constant function on $\RR^d$ with value $u(X)$:
\begin{equation}
(\mathcal{Z}_{X}u)(x)\equiv u(X).\label{eq:Zdef}
\end{equation}
In the following statement, we compare two solutions over $[0,t]$ with the noise turned off on $[\tau_1,\tau_2] \subset [0,t]$, but 
one of them additionally has the solution at time $\tau_2$ replaced by its value at some fixed location $X$ before the evolution 
of the SPDE continues. %

Recall the notation $\mathscr{X}_r$ from \cref{sec:notation}.
\begin{prop}
    \label{prop:flatten}For every $T>0$ and $\lambda<\sqrt{2(d-2)}$, there is a
$C=C(T,\lambda)<\infty$ such that the following holds.
Suppose that $\Lip(\sigma)<\lambda$.
Let $\tau_1<\tau_2$, $u_{\tau_1}\in\mathscr{X}_{\tau_1}$,
and $X\in\RR^{d}$. Then we have, for each $x\in\RR^{d}$, that 
\[
\mathbb{E}\lvert\mathcal{U}_{\tau_2,t}(\mathcal{Z}_{X}\mathcal{G}_{\tau_2-\tau_1}u_{\tau_1})(x)-\mathcal{U}_{\tau_2,t}(\mathcal{G}_{\tau_2-\tau_1}u_{\tau_1})(x)\rvert^2
\le C\langle\vvvert u_{\tau_1}\vvvert\rangle^2\frac{t-\tau_2+\lvert x-X\rvert^2}{\tau_2-\tau_1}.
\]
\end{prop}

\begin{proof}
Let $\tilde{u}_t=\mathcal{U}_{\tau_2,t}\mathcal{G}_{\tau_2-\tau_1}u_{\tau_1}$
and $\overline{u}_t=\mathcal{U}_{\tau_2,t}\mathcal{Z}_{X}\mathcal{G}_{\tau_2-\tau_1}u_{\tau_1}$.
We can write
\begin{align*}
\tilde{u}_t(x)-\overline{u}_t(x) & =\mathcal{G}_{t-\tau_1}u_{\tau_1}(x)-\mathcal{G}_{\tau_2-\tau_1}u_{\tau_1}(X)\\
 & \qquad+\frac1{\sqrt{\log\rho^{-1}}}\int_{\tau_2}^t\int\clG_{t-s}[(\sigma(\tilde{u}_{s})-\sigma(\overline{u}_{s}))\,\dif W_{s}^{\rho}](x).
\end{align*}
This means that, if we define $f(t,x)=\EE\lvert\tilde{u}_t(x)-\overline{u}_t(x)\rvert^2$,
then we have
\begin{align*}
f(t,x) & \le\mathbb{E}\left\lvert\mathcal{G}_{t-\tau_1}u_{\tau_1}(x)-\mathcal{G}_{\tau_2-\tau_1}u_{\tau_1}(X)\right\rvert^2\\
 & \qquad+\frac{\Lip(\sigma)^2}{\log\rho^{-1}}\int_{\tau_2}^t\iint R^{\rho}(y_1-y_2)\left(\prod_{i=1}^2G_{t-s}(x-y_{i})f(s,y_{i})^{1/2}\right)\,\dif y_1\,\dif y_2\,\dif s.
\end{align*}
We can estimate
\begin{align*}
\mathbb{E} & \left\lvert\mathcal{G}_{t-\tau_1}u_{\tau_1}(x)-\mathcal{G}_{\tau_2-\tau_1}u_{\tau_1}(X)\right\rvert^2=\mathbb{E}\left\lvert\int\left(G_{t-\tau_1}(x-y)-G_{\tau_2-\tau_1}(X-y)\right)u_{\tau_1}(y)\,\dif y\right\rvert^2\\
 & \le\vvvert u_{\tau_1}\vvvert\cdot\|G_{t-\tau_1}(x-\cdot)-G_{\tau_2-\tau_1}(X-\cdot)\|_{L^1(\mathbb{R})}^2\\
 & \le\vvvert u_{\tau_1}\vvvert\cdot\frac{2(t-\tau_2)+\lvert x-X\rvert^2}{\tau_2-\tau_1},
\end{align*}
with the last inequality by \cite[(5.6--5.7)]{DG22}. We then conclude
by \cref{lem:spatial-induction} below.
\end{proof}
\begin{lem}
\label{lem:spatial-induction}For every $\omega<1$, we have a constant
$C=C(\omega,d)<\infty$ such that, if $f$ is a positive bounded function and $A_1,A_2\in (0,\infty)$ are constants such
that
\begin{equation}
\begin{aligned}f(t,x) & \le A_1+A_2\lvert x-X\rvert^2\\
 & \qquad+\frac{\lambda^2}{\log\rho^{-1}}\int_{\tau_2}^t\iint R^{\rho}(y_1-y_2)\prod_{i=1}^2G_{t-s}(x-y_{i})f(s,y_{i})^{1/2}\,\dif y_1\,\dif y_2\,\dif s
\end{aligned}
\label{eq:fbound}
\end{equation}
and
\begin{equation}
\frac{\lambda^2\log\frac{t-\tau_2+\rho}{\rho}}{2(d-2)\log\rho^{-1}}\le\omega,\label{eq:flattenbd-geomcond}
\end{equation}
then, for all $t\ge\tau_2$, we have
\begin{equation}
f(t,x)\le C\left(A_1+A_2(t-\tau_2+\lvert x-X\rvert^2)\right).\label{eq:spatial-induction-conclusion}
\end{equation}
\end{lem}

\begin{proof}
Since $f$ is assumed to be bounded, we know that, for some values of $B_1$ and $B_2$,
\begin{equation}
f(t,x)\le B_1+B_2\lvert x-X\rvert^2.\label{eq:fansatz}
\end{equation}
In view of \cref{eq:fbound}, the inequality
\cref{eq:fansatz} implies that
\begin{equation}
\begin{aligned}f(t,x) & \le A_1+A_2\lvert x-X\rvert^2\\
 & +\frac{\lambda^2}{\log\rho^{-1}}\int_{\tau_2}^t\iint R^{\rho}(y_1-y_2)\prod_{i=1}^2G_{t-s}(x-y_{i})\left(B_1+B_2\lvert y_{i}-X\rvert^2\right)^{1/2}\,\dif y_1\,\dif y_2\,\dif s.
\end{aligned}
\label{eq:fuseansatz}
\end{equation}
We make the change of variables
\begin{equation}
w=\frac12(y_1+y_2),\qquad z=y_1-y_2.\label{eq:chgvar}
\end{equation}
We note that
\begin{align*}
\prod_{i=1}^2\left(B_1+B_2\lvert y_{i}-X\rvert^2\right) & =\left(B_1+B_2\lvert w+z/2-X\rvert^2\right)\left(B_1+B_2\lvert w-z/2-X\rvert^2\right)\\
 & =B_1^2+B_1B_2\left(\lvert w-X\rvert^2+\frac1{4}\lvert z\rvert^2\right)\\
 & \qquad+B_2^2\left(\left(\lvert w-X\rvert^2+\frac1{4}\lvert z\rvert^2\right)^2-\frac12\left((w-X)\cdot z\right)^2\right)\\
 & \le\left(B_1+B_2\left(\lvert w-X\rvert^2+\frac1{4}\lvert z\rvert^2\right)\right)^2,
\end{align*}
so
\begin{equation}
\prod_{i=1}^2\left(B_1+B_2\lvert y_{i}-X\rvert^2\right)^{1/2}\le B_1+B_2\left(\lvert w-X\rvert^2+\frac1{4}\lvert z\rvert^2\right).\label{eq:Bprodineq}
\end{equation}
Using the change of variables \cref{eq:chgvar} and the estimate
\cref{eq:Bprodineq} in \cref{eq:fuseansatz}, we obtain
\begin{align}
f(t,x) & \le A_1+A_2\lvert x-X\rvert^2\nonumber \\
       & \qquad+\frac{\lambda^2}{\log\rho^{-1}}\int_{\tau_2}^t\iint R^{\rho}(z)\left(B_1+B_2\left(\lvert w-X\rvert^2+\frac1{4}\lvert z\rvert^2\right)\right)\nonumber \\
 & \qquad\qquad\qquad\qquad\qquad\times G_{t-s}(x-w-z/2)G_{t-s}(x-w+z/2)\,\dif w\,\dif z\,\dif s\nonumber \\
 & =A_1+A_2\lvert x-X\rvert^2\nonumber \\
 & \qquad+\frac{\lambda^2}{\log\rho^{-1}}\int_{\tau_2}^t\iint R^{\rho}(z)\left(B_1+B_2\left(\lvert w-X\rvert^2+\frac1{4}\lvert z\rvert^2\right)\right)\nonumber \\
 & \qquad\qquad\qquad\qquad\qquad\times G_{2(t-s)}(z)G_{\frac{t-s}2}(x-w)\,\dif w\,\dif z\,\dif s.\label{eq:chgvarinftx}
\end{align}
Now we have
\begin{align}
\frac{\lambda^2}{\log\rho^{-1}} & \int_{\tau_2}^t\iint\left(B_1+B_2\lvert w-X\rvert^2\right)R^{\rho}(z)G_{2(t-s)}(z)G_{\frac{t-s}2}(x-w)\,\dif w\,\dif z\,\dif s\nonumber \\
\overset{\cref{eq:HKwithkernelat0}}&{=}\frac{\lambda^2}{2(d-2)\log\rho^{-1}}\int_{\tau_2}^t\frac{B_1+B_2\int G_{\frac{t-s}2}(x-w)\lvert w-X\rvert^2\,\dif w}{t-s+\rho}\,\dif s\nonumber \\
 & =\frac{\lambda^2}{2(d-2)\log\rho^{-1}}\int_{\tau_2}^t\frac{B_1+B_2\left(d(t-s)/2+\lvert x-X\rvert^2\right)}{t-s+\rho}\,\dif s\nonumber \\
\overset{\cref{eq:flattenbd-geomcond}}&{\le}\omega\left(B_1+B_2\left(d(t-\tau_2)/2+\lvert x-X\rvert^2\right)\right).\label{eq:easyterms}
\end{align}
We also have that
\begin{align}
\frac{\lambda^2}{4\log\rho^{-1}} & \int_{\tau_2}^t\iint R^{\rho}(z)\lvert z\rvert^2G_{2(t-s)}(z)G_{\frac{t-s}2}(x-w)\,\dif w\,\dif z\,\dif s\le\frac{C\lambda^2(t-\tau_2)}{\log\rho^{-1}}\label{eq:hardterms}
\end{align}
by \cref{eq:Rrhobound}. Using \cref{eq:easyterms} and
\cref{eq:hardterms} in \cref{eq:chgvarinftx}, we get
\begin{align*}
f(t,x) & \le A_1+\left(B_1+B_2d(t-\tau_2)/2\right)\omega+CB_2\frac{\lambda^2(t-\tau_2)}{\log\rho^{-1}}+\left(A_2+\omega B_2\right)\lvert x-X\rvert^2.
\end{align*}

Therefore, if we take $B_2^{(0)}=0$ and $B_1^{(0)}$ so large
that \cref{eq:fansatz} holds with $B_{i}\mapsfrom B_{i}^{(0)}$,
and then inductively define
\begin{align}
B_1^{(n)} & =A_1+\omega\left(B_1^{(n-1)}+B_2^{(n-1)}d(t-\tau_2)/2\right)+\frac{C\lambda^2(t-\tau_2)}{\log\rho^{-1}}B_2^{(n-1)};\label{eq:B1n}\\
B_2^{(n)} & =A_2+\omega B_2^{(n-1)},\label{eq:B2n}
\end{align}
then we have
\[
f(t,x)\le B_1^{(n)}+B_2^{(n)}\lvert x-X\rvert^2\qquad\text{for each }n\in\mathbb{N}.
\]

Now from \cref{eq:B2n} and the assumption that $\omega<1$ we
see that
\[
\limsup_{n\to\infty}B_2^{(n)}\le\frac{A_2}{1-\omega},
\]
and using this in \cref{eq:B1n}, we see that
\[
\limsup_{n\to\infty}B_1^{(n)}\le\frac{A_1+\frac{A_2}{1-\omega}\left(d/2+C\lambda^2/\log\rho^{-1}\right)(t-\tau_2)}{1-\omega}.
\]
The last three displays together imply \cref{eq:spatial-induction-conclusion}.
\end{proof}

We summarize the results of this section in the following combined proposition.

\begin{prop}
\label{prop:turnoffflatten}For every $T>0$ and $\lambda<\sqrt{2(d-2)}$,
there is a $C=C(T,\lambda)<\infty$ such that the following holds.
Suppose that $\Lip(\sigma)<\lambda$. Let $\tau_1<\tau_2\in\RR$
be such that $\tau_2\le\tau_1+T$, and suppose that $\rho\in(0,C^{-1})$
and $u_{\tau_1}\in\mathscr{X}_{\tau_1}$. Then we have
\begin{equation}
\begin{aligned} & \left(\mathbb{E}\left\lvert\mathcal{U}_{\tau_2,t}\mathcal{Z}_{X}\mathcal{G}_{\tau_2-\tau_1}u_{\tau_1}(x)-\mathcal{U}_{\tau_1,t}u_{\tau_1}(x)\right\rvert^2\right)^{1/2}\\
 & \qquad\le C\langle\vvvert u_{\tau_1}\vvvert\rangle\left[\left(\frac{\log\left(\frac{t-\tau_1+\rho}{t-\tau_2+\rho}\right)+1}{\log\rho^{-1}}\right)^{1/2}+\left(\frac{t-\tau_2+\lvert x-X\rvert^2}{\tau_2-\tau_1}\right)^{1/2}\right].
\end{aligned}
\label{eq:turnoffflatten}
\end{equation}
\end{prop}

\begin{proof}
This follows from combining \cref{prop:turnoffnoise,prop:flatten} using the triangle inequality.
\end{proof}

\section{Decorrelation estimate\label{sec:cov-bd}}

Since, unlike \cite{DG22}, we consider noise with long range correlations, the statistics of the solution involve covariances 
of $\sigma(u)$ at distant locations, rather than at locations that are very close to one another. The following decorrelation 
result will be essential for bounding such covariances.
\begin{prop}
\label{prop:covbd}For each $T>0$, there is a $C=C(T,\sigma,d)\in(1,\infty)$
such that the following holds. Let $f\in\Lip(\RR^{\mathfrak{m}})$ and let $u_t$ solve (\ref{eq:ueqn}) with initial data $u_{0}\in\mathscr{X}_{0}$.
Then we have, for all $\rho\in(0,C^{-1})$, that
\begin{equation}
\Cov(f(u_t(x_1)),f(u_t(x_2)))\le\frac{C\Lip(f)^2\langle \vvvert u_0\vvvert\rangle^2}{\log\rho^{-1}}\left(\log\frac{t^{1/2}}{\lvert x_1-x_2\rvert}+1+\frac{\lvert x_1-x_2\rvert}{t^{1/2}}\right).\label{eq:covbd}
\end{equation}
\end{prop}

\begin{proof}
We assume that $f\colon\mathbb{R}^{\mathfrak{m}}\to\mathbb{R}$ is
smooth; the case of general Lipschitz $f$ follows by approximation. Using the Clark--Ocone
formula, we can write
\begin{equation}
\begin{aligned}f(u_t(x))-\EE[f(u_t(x))] & =\int_{0}^t\int\EE[\Dif_{r,z}[f(u_t(x))]\mid\clF_{s}]\,\dif W_{s}^{\rho}(z)\\
 & =\int_{0}^t\int\EE[\nabla f(u_t(x))\cdot\Dif_{r,z}u_t(x)\mid\clF_{s}]\,\dif W_{s}^{\rho}(z).
\end{aligned}
\label{eq:applyCO}
\end{equation}
We then estimate the covariance by
\begin{align*}
\Cov & (f(u_t(x_1)),f(u_t(x_2)))\\
 & \le\int_{0}^t\iint R^{\rho}(z_1-z_2)\EE\left[\prod_{i=1}^2\left\lvert\mathbb{E}[\nabla f(u_t(x_{i}))\cdot\Dif_{r,z_{i}}u_t(x_{i})\mid\clF_{r}]\right\rvert\right]\,\dif z_1\,\dif z_2\,\dif r\\
 & \le\Lip(f)^2\int_{0}^t\iint R^{\rho}(z_1-z_2)\prod_{i=1}^2\left(\EE\lvert\Dif_{r,z_{i}}u_t(x)\rvert_{\Frob}^2\right)^{1/2}\,\dif z_1\,\dif z_2\,\dif r.
\end{align*}
Applying \cref{eq:MDbound} below, we obtain
\begin{align*}
    \Cov&(f(u_t(x_1)),f(u_t(x_2))) \\& \le\frac{C\Lip(f)^2\langle\vvvert u_0\vvvert\rangle^2}{\log\rho^{-1}}\int_{0}^t\iint R^{\rho}(z_1-z_2)\prod_{i=1}^2G_{t-r}(x_{i}-z_{i})\,\dif z_1\,\dif z_2\,\dif r,
\end{align*}
so that \cref{eq:covbd} then follows from \cref{eq:integrateHKRrho}.
\end{proof}
In the proof of \cref{prop:covbd}, we needed a bound on the
Malliavin derivative norm $\EE\lvert\Dif_{r,z_{i}}u_t(x)\rvert_{\Frob}^2$:
\begin{prop}
\label{prop:MDbd}For each $T\in(0,\infty)$, $M,\beta\in (0,\infty)$, and $\omega<1$, there is a constant $C=C(T,\omega,M,\beta)$ such that the following holds.
Suppose that $\lvert \sigma(u)\rvert_{\Frob}^2\le M+\beta^2\lvert u\rvert^2$ for all $u\in\RR^{\mathfrak{m}}$ and that $\frac{\Lip(\sigma)^2}{d-2}\log_{\rho^{-1}}(1+\frac32\rho^{-1}T)\le\omega$.
Let $u_{0}\in\mathscr{X}_{0}$ and let $u_t=\mathcal{U}_{0,t}u_{0}$.
Then we have
\begin{equation}
\mathbb{E}\lvert\Dif_{r,z}u_t(x)\rvert_{\Frob}^2\le\frac{C\langle \vvvert u_0\vvvert\rangle^2 G_{t-r}^2(x-z)}{\log\rho^{-1}}.\label{eq:MDbound}
\end{equation}
\end{prop}

\begin{proof}
Define, for $t\ge0$,
$u_t^{0}(x)=u_{0}(x)$ and, for $n\ge1$,
\begin{equation}
    u_t^{n}(x)=\mathcal{G}_tu_{0}(x)+\frac1{\sqrt{\log\rho^{-1}}}\int_{0}^t\clG_{t-s}[\sigma(u_{s}^{n-1})\,\dif W_{s}^{\rho}](x).\label{eq:unmild}
\end{equation}
It is standard that $(u^n_t)$ converges to $(u_t)$ (for fixed $\rho>0$, in probability, uniformly on compact sets) as $n\to \infty$.
For $0<r<t$, and $z,x\in\RR^2$, we can take Malliavin derivatives in \cref{eq:unmild} %
to obtain
\begin{align*}
\Dif_{r,z}u_t^{n}(x) & =\frac1{\sqrt{\log\rho^{-1}}}G_{t-r}(x-z)\sigma(u_{s}^{n-1}(z))\\
 & \qquad+\frac1{\sqrt{\log\rho^{-1}}}\int_{r}^t\int G_{t-s}(x-z)\underline{\nabla\sigma(u_{s}^{n-1}(z))}\cdot\Dif_{r,z}u_{s}^{n-1}(z)\,\dif W_{s}^{\rho}(z),
\end{align*}
where $(\underline{\nabla\sigma(u_{s}^{n-1}(z))})$ is an adapted process satisfying
\[
    \sup_{s,z}\|\underline{\nabla\sigma(u_{s}^{n-1}(z))}\|_{\Frob}\le \Lip(\sigma)\qquad\text{almost surely.}
\]
In the case when $\sigma$ is in fact continuously differentiable, then we can indeed take $\underline{\nabla\sigma(u_{s}^{n-1}(z))} = \nabla\sigma(u_{s}^{n-1}(z))$. If this is not the case, then the statement nonetheless holds, as can be seen be a simple adaptation of the proof of 
\cite[Proposition~1.2.4]{Nua06}. That proposition concerns scalar-valued functions and finite-dimensional noise, but the same argument holds for vector-valued functions, and the infinite-dimensional noise we consider can be handled by approximation.

This having been established, we can continue by writing for all $n\ge1$ (using also the assumption on $\sigma$),
\begin{align}
\mathbb{E} & \lvert\Dif_{r,z}u_t^{n}(x)\rvert_{\Frob}^2\nonumber \\
           & \le\frac{M+\beta^2\sup_{s\in [0,T]}\vvvert u_{s}^{n-1}\vvvert^2}{\log\rho^{-1}}G_{t-r}^2(x-z)\nonumber \\
 & \quad+\frac{\Lip(\sigma)^2}{\log\rho^{-1}}\int_{r}^t\iint R^{\rho}(y_1-y_2)\prod_{i=1}^2\left(G_{t-s}(x-y_{i})\left(\mathbb{E}\left\lvert\Dif_{r,z}u_{s}^{n-1}(y_{i})\right\rvert_{\Frob}^2\right)^{1/2}\right)\,\dif y_1\,\dif y_2\,\dif s.\label{eq:Difurecurrence}
\end{align}
Moreover, we have
\[
\mathbb{E}\lvert\Dif_{r,z}u_t^{0}(x)\rvert_{\Frob}^2=0
\]
since $u_t^{0}(x)=u_{0}(x)$ is deterministic.

Now we can apply \cref{prop:MDinduction} below with
\[
\lambda_1=M+\beta^2\sup_{s\in [0,T]}\vvvert u_{s}^{n-1}\vvvert^2\qquad\text{and}\qquad\lambda_2=\Lip(\sigma)^2,
\]
and time shifted by $r$,
and note that $\lambda_1\le C\langle\vvvert u_0\vvvert\rangle^2$ by an easy modification of \cref{prop:momentbd} (to apply to the $u_n$s rather than to $u$), to
obtain
\[
\mathbb{E}\lvert\Dif_{r,z}u_t^{n}(x)\rvert_{\Frob}^2\le\frac{C\langle \vvvert u_0\vvvert\rangle^2 G_{t-r}^2(x-z)}{\log\rho^{-1}}.
\]
Passing to the limit as $n\to\infty$,
we obtain \cref{eq:MDbound}.
\end{proof}
The following proposition is an analogue of \cite[Lemma 2.7]{GHL23}
that allows us to analyze the recurrence \cref{eq:Difurecurrence}. 
\begin{prop}
\label{prop:MDinduction}For each $\omega<1$ and $\lambda_1,\lambda_2,T\in(0,\infty)$,
there is a constant $C=C(\omega,\lambda_1,\lambda_2,T)$ such
that the following holds. Suppose that $g_{n}\colon[0,T]\times\RR^{d}\to\RR$
are such that
\[
g_{0}(t,x)=0\qquad\text{for all }t\in[0,T]\text{ and }x\in\RR^{d}
\]
and we have $\lambda_1,\lambda_2,T\ge0$ such that
\[
g_{n+1}^2(t,x)\le\frac1{\log\rho^{-1}}\left(\lambda_1G_t^2(x)+\lambda_2\int_{0}^tJ_t(s,x)[g_{n}]\,\dif s\right)\quad\text{for all }x\in\mathbb{R}^{d}\text{ and }t\in[0,T],
\]
where we have defined
\[
J_t^{\rho}(s,x)[g]=\iint R^{\rho}(y_1-y_2)\prod_{i=1}^2\left(G_{t-s}(x-y_{i})g(s,y_{i})\right)\,\dif y_1\,\dif y_2.
\]
Then, as long as 
\begin{equation}
\frac{\lambda_2}{d-2}\cdot\frac{\log(1+\frac32T\rho^{-1})}{\log\rho^{-1}}\le\omega,\label{eq:MDinductionomegacondition}
\end{equation}
we have
\[
g_{n}(t,x)\le\frac{CG_t(x)}{\log\rho^{-1}}\qquad\text{for all }x\in\mathbb{R}^{d}\text{ and }t\in[0,T].
\]
\end{prop}

\begin{proof}
We define
\[
f_{0}(t,x)=0
\]
and
\begin{equation}
f_{n}^2(t,x)=\frac1{\log\rho^{-1}}\left(\lambda_1G_t^2(x)+\lambda_2\int_{0}^tJ_t(s,x)[f_{n-1}]\,\dif s\right)\;,\label{eq:fn2bd}
\end{equation}
so that $g_{n}\le f_{n}$ for all $n$.

We then look for an inductive choice of $H_n$ with $H_0 = 0$ and such that
\begin{equation}
f_{n}^2(t,x)\le\frac{H_{n}}{\log\rho^{-1}}G_t^2(x),\label{eq:fninduction}
\end{equation}
We have
\begin{align*}
\int_{0}^t & J_t(s,x)[f_{n}]\,\dif s\\
 & =\int_{0}^t\iint G_{t-s}(x-y_1)G_{t-s}(x-y_2)R^\rho(y_1-y_2)f_{n}(s,y_1)f_{n}(s,y_2)\,\dif y_1\,\dif y_2\,\dif s\\
 & \le\frac{H_{n}}{\log\rho^{-1}}\int_{0}^t\iint R^{\rho}(y_1-y_2)\prod_{i=1}^2\left(G_{t-s}(x-y_{i})G_{s}(y_{i})\right)\,\dif y_1\,\dif y_2\,\dif s\\
\overset{\cref{eq:decomposeRrho}}&{=}\frac{(2\pi)^{d/2}H_{n}}{2\log\rho^{-1}}\int_{0}^t\iint\left(\prod_{i=1}^2\left(G_{t-s}(x-y_{i})G_{s}(y_{i})\right)\right)\int_{0}^{\infty}\nu^{d/2-2}G_{\nu+2\rho}(y_1-y_2)\\&\qquad\qquad\qquad\qquad\qquad\qquad\qquad\qquad\qquad\qquad\qquad\qquad\qquad\,\dif\nu\,\dif y_1\,\dif y_2\,\dif s.
\end{align*}
We use the fact that
\[
G_{t-s}(x-y)G_{s}(y)=G_t(x)G_{\frac{s}t(t-s)}\left(y-\frac{s}tx\right)
\]
to turn this into
\begin{align*}
\int_{0}^t & J_t(s,x)[f_{n}]\,\dif s\\
             & \le\frac{(2\pi)^{d/2}H_{n}}{2\log\rho^{-1}}G_t(x)^2\int_{0}^t\iint\left(\prod_{i=1}^2G_{\frac{s}t(t-s)}\left(y_{i}-\frac{s}tx\right)\right)\int_{0}^{\infty}\nu^{d/2-2}G_{\nu+2\rho}(y_1-y_2)\\&\qquad\qquad\qquad\qquad\qquad\qquad\qquad\qquad\qquad\qquad\qquad\qquad\qquad\,\dif\nu\,\dif y_1\,\dif y_2\,\dif s\\
 & =\frac{(2\pi)^{d/2}H_{n}}{2\log\rho^{-1}}G_t(x)^2\int_{0}^t\int_{0}^{\infty}\nu^{d/2-2}G_{2\frac{s}t(t-s)+\nu+2\rho}(0)\,\dif\nu\,\dif s\\
 & =\frac{H_{n}}{2\log\rho^{-1}}G_t^2(x)\int_{0}^t\int_{0}^{\infty}\frac{\nu^{d/2-2}}{(2\frac{s}t(t-s)+\nu+2\rho)^{d/2}}\,\dif\nu\,\dif s\\
 & =\frac{H_{n}G_t^2(x)}{(d-2)\log\rho^{-1}}\int_{0}^t\frac{\dif s}{2\frac{s}t(t-s)+2\rho}.%
\end{align*}
We estimate the last integral as
\[
    \int_{0}^t\frac{\dif s}{2\frac{s}t(t-s)+2\rho}
  = \frac{t}{\sqrt{t(t+4\rho)}}\log\frac{t+\sqrt{t(4\rho+t)}+2\rho}{2\rho}
  \le \log\frac{3t+2\rho}{2\rho}.
\]
Using the last two displays in \cref{eq:fn2bd}, we obtain the bound
\begin{align*}
f_{n}^2(t,x) %
 & \le\frac1{\log\rho^{-1}}\left(\lambda_1+\lambda_2\frac{H_{n}}{d-2}\frac{\log(1+\frac32t\rho^{-1})}{\log\rho^{-1}}\right)G_t^2(x)\\
\overset{\cref{eq:MDinductionomegacondition}}&{\le}\frac1{\log\rho^{-1}}\left(\lambda_1+\omega H_{n}\right)G_t^2(x)\;.
\end{align*}
It follows that the choice $H_{n+1}=\lambda_1+\omega H_{n}$
allows us to satisfy the bound \cref{eq:fninduction} for all $n$. Since $H_{0}=0$ and $\omega<1$,
we find that $H_{n} < \frac{\lambda_1}{1-\omega}$
for all $n$, completing the proof.
\end{proof}

\section{The approximate decoupling function\label{sec:approx-decoupling}} %

\subsection{Definition and regularity}
We define the matrix-valued function
\begin{equation}
    L_{\rho}(q,a)=\EE[\sigma(\mathcal{U}_{0,\rho^{1-q}}a(x))]\qquad\text{for }a\in\RR^\fkm\text{ and }q\ge0.\label{eq:Lrhodef}
\end{equation}
The quantity $J_\rho \coloneqq L_\rho/\sqrt{2(d-2)}$ (see \cref{eq:Jrhodef} below) will be an approximation of the decoupling function $J$ appearing in \cref{eq:FBSDE-J}. Establishing this will be the goal of the next two sections.
We begin with some temporal and spatial regularity properties for $L_{\rho}$. These will in particular give us a compactness statement for $L_\rho$ in \cref{prop:Lcompactness} below.
\begin{lem}
\label{lem:lem:Lbdd}There is a constant $C=C(\sigma)\in(1,\infty)$
such that we have, for all $q\in[0,1]$ and all $\rho\in(0,C^{-1})$,
that
\begin{equation}
L_{\rho}(q,a)\le C\langle a\rangle.\label{eq:decouplingbdd}
\end{equation}
\end{lem}

\begin{proof}
This is a consequence of \cref{prop:momentbd} and the fact that
$\sigma$ is Lipschitz.
\end{proof}
\begin{lem}
\label{lem:Ltimereg}There is a constant $C\in(1,\infty)$ such that
we have, for all $q_{1},q_{2}\in[0,1]$ and all $\rho\in(0,C^{-1})$,
that
\begin{equation}
\lvert L_{\rho}(q_{1},a)-L_{\rho}(q_{2},a)\rvert_{\mathrm{F}}\le C\Lip(\sigma)\langle a\rangle\left(\lvert q_{2}-q_{1}\rvert^{1/2}+\frac{1}{\sqrt{\log\rho^{-1}}}\right).\label{eq:Ltimereg}
\end{equation}
\end{lem}
\begin{proof}
Assume without loss of generality that $q_{2}\ge q_{1}$. Then we have
\begin{align*}
\lvert L_{\rho}(q_{2},a)-L_{\rho}(q_{1},a)\rvert_{\mathrm{F}} & \le\EE\lvert\sigma(\mathcal{U}_{0,\rho^{1-q_{2}}}a(x))-\sigma(\mathcal{U}_{\rho^{1-q_{2}}-\rho^{1-q_{1}},\rho^{1-q_{2}}}a(x))\rvert_{\mathrm{F}}\\
 & \le\Lip(\sigma)\EE\lvert\mathcal{U}_{0,\rho^{1-q_{2}}}a(x)-\mathcal{U}_{\rho^{1-q_{2}}-\rho^{1-q_{1}},\rho^{1-q_{2}}}a(x)\rvert\\
 & \le\Lip(\sigma)\left(\EE\lvert\mathcal{U}_{0,\rho^{1-q_{2}}}a(x)-\mathcal{U}_{\rho^{1-q_{2}}-\rho^{1-q_{1}},\rho^{1-q_{2}}}\clG_{\rho^{1-q_{1}}}a(x)\rvert^{2}\right)^{1/2}\\
 & \le C\Lip(\sigma)\langle a\rangle\left(\frac{\log\left(\frac{\rho^{1-q_{2}}+\rho}{\rho^{1-q_{1}}+\rho}\right)+1}{\log\rho^{-1}}\right)^{1/2}\\
 & \le C\Lip(\sigma)\langle a\rangle\left(\lvert q_{2}-q_{1}\rvert^{1/2}+\frac{1}{\sqrt{\log\rho^{-1}}}\right),
\end{align*}
with the penultimate inequality following from \cref{prop:turnoffnoise,prop:momentbd}.
\end{proof}
\begin{lem}
\label{lem:Lspacereg}There exists a constant $C<\infty$ such that, for
all $q\in[0,1]$ and all $a_{1},a_{2}\in\RR^{\mathfrak{m}}$,
\begin{equation}
\lvert L_{\rho}(q,a_{1})-L_{\rho}(q,a_{2})\rvert_{\mathrm{F}}\le C\Lip(\sigma)\lvert a_{1}-a_{2}\rvert.\label{eq:Lspacereg}
\end{equation}
\end{lem}

\begin{proof}
We have
\begin{align}
\lvert L_{\rho}(q,a_{1})-L_{\rho}(q,a_{2})\rvert_{\mathrm{F}} & \le\EE\lvert\sigma(\mathcal{U}_{0,\rho^{1-q}}a_{1}(x))-\sigma(\mathcal{U}_{0,\rho^{1-q}}a_{2}(x))\rvert_{\mathrm{F}}\nonumber \\
 & \le\Lip(\sigma)\EE\lvert \mathcal{U}_{0,\rho^{1-q}}a_{1}(x)-\mathcal{U}_{0,\rho^{1-q}}a_{2}(x)\rvert.\label{eq:Jspaceregstart}
\end{align}
Now we perform an $L^{2}$ analysis along the lines of \cite[Proposition 3.3]{DG22}.
We have, defining $v_{t}^{(i)}(x)=\mathcal{U}_{0,\rho^{1-q}}a_{i}(x)$,
that
\begin{align*}
\EE & \lvert v_{t}^{(1)}(x)-v_{t}^{(2)}(x)\rvert_{\mathrm{F}}^{2}\\
 & \le\lvert a_{1}-a_{2}\rvert^{2}+\frac{\Lip(\sigma)^{2}}{\log\rho^{-1}}\int_{0}^{\rho^{1-q}}\vvvert v_{s}^{(1)}-v_{s}^{(2)}\vvvert^{2}\iint R^{\rho}(y_{1}-y_{2})\prod_{i=1}^{2}G(x-y_{1})\,\dif y_{1}\,\dif y_{2}\,\dif s\\
\overset{\cref{eq:HKwithkernelat0}}&{=}\lvert a_{1}-a_{2}\rvert^{2}+\frac{\Lip(\sigma)^{2}}{2(d-2)\log\rho^{-1}}\int_{0}^{\rho^{1-q}}\frac{\vvvert v_{s}^{(1)}-v_{s}^{(2)}\vvvert^{2}}{\rho^{1-q}-s+\rho}\,\dif s.
\end{align*}
Then by \cite[Lemma 3.4]{DG22} (with constants changed), we have
\[
\vvvert v_{t}^{(1)}-v_{t}^{(2)}\vvvert^{2}\le C\lvert a_{1}-a_{2}\rvert^{2},
\]
and using this in \cref{eq:Jspaceregstart} we obtain \cref{eq:Lspacereg}.
\end{proof}
The regularity properties of $L_{\rho}$ just established give us
compactness:
\begin{prop}
\label{prop:Lcompactness}For any sequence $\rho_{k}\downarrow0$, there
is a subsequence $\rho_{k_{i}}\downarrow0$ and a continuous function
$L\colon[0,1]\times\RR^{\mathfrak{m}}\to\RR^\fkm\otimes\RR^\fkn$ such
that
\[
\lim_{i\to\infty}L_{\rho_{k_{i}}}|_{[0,1]\times\RR^{\mathfrak{m}}}=L,
\]
uniformly on compact subsets of $[0,1]\times\RR^{\mathfrak{m}}$.
\end{prop}

\begin{proof}
The proof is by compactness as in the proof of \cite[Proposition 7.4]{DG22}.
The only difference is that, since we do not assume that $\sigma(0)=0$,
we also need the boundedness result \cref{lem:lem:Lbdd} to apply
Arzelà--Ascoli.
Since the proof is otherwise the same, we omit the details.
\end{proof}

\subsection{Statistics of the averaged fields\label{sec:statistics}}
In this section we relate the approximate decoupling function $L_\rho$ to spatial averages of the solution run on appropriate time scales.
For a random vector $u\in\RR^\fkm$, we use the notation $\Var(u) \coloneqq \EE[u^{\otimes 2}]\coloneqq \EE[uu^\top]$.
\begin{prop}\label{prop:VarofStep}
    For each $T\in (0,\infty)$ and $\lambda<\sqrt{d-2}$, we have a constant $C=C(\sigma,T,\lambda,d)<\infty$ such that, whenever $\Lip(\sigma)<\lambda$, then for any $0\le \tau_1\le\tau_2\le T$, $a\in\RR^\fkm$, and $x\in\RR^d$, we have
    \begin{equation}\label{eq:VarofStep}
    \begin{aligned}
    &\langle a\rangle^{-2}\left\lvert \Var(\clG_{\tau_2-\tau_1}\clU_{0,\tau_1}a(X)) - \frac{L_\rho(1-\log_\rho\tau_1,a)^{\otimes 2}\log_{\rho^{-1}}\left(1+\frac{\tau_1}{2(\tau_2-\tau_1+\rho)}\right)}{2(d-2)}\right\rvert_\Frob
  \\&\ \le C\log_{\rho^{-1}}\frac{\tau_2+\rho}{\tau_2-\frac12\tau_1+\rho}+C\left(\frac1{\sqrt{\log\rho^{-1}}}+\log_{\rho^{-1}}\frac{\tau_1}{\tau_2-\tau_1}+\frac{\sqrt{\tau_2/\tau_1}}{\log\rho^{-1}}\right)\log_{\rho^{-1}}\frac{\tau_2-\frac12\tau_1}{\tau_2-\tau_1}.
    \end{aligned}
\end{equation}
    If in addition we assume that $\tau_2\le2\tau_1$, then in particular we have the inequality
    \begin{equation}\label{eq:varofstep-simple}
        \begin{aligned}
    &\left\lvert \Var(\clG_{\tau_2-\tau_1}\clU_{0,\tau_1}a(X)) - \frac{L_\rho(1-\log_\rho\tau_1,a)^{\otimes 2}\log_{\rho^{-1}}\left(1+\frac{\tau_1}{2(\tau_2-\tau_1+\rho)}\right)}{2(d-2)}\right\rvert_\Frob
          \\& \qquad\le C\langle a\rangle^2\left(\frac1{\log\rho^{-1}}+\left(\log_{\rho^{-1}}\left(1+\frac{\tau_1}{\tau_2-\tau_1}\right)\right)^2\right).
        \end{aligned}
    \end{equation}
\end{prop}
\begin{proof}
    Let $u_s = \clU_{0,s}a$.
    We have 
    \begin{equation}
        \begin{aligned}
          \Var(\clG_{\tau_2-\tau_1}u_{\tau_1}(X))&=\Var\left(\frac1{\sqrt{\log\rho^{-1}}}\int_0^{\tau_1}\int G_{\tau_2-s}(X-y)\sigma(u_s(y))\,\dif W^\rho_s(y)\right)\\&=\int_0^{\tau_1} A(s)\,\dif s,
\end{aligned}\label{eq:vardecomp}
    \end{equation}
    where we have defined
\begin{align*}
    A(s)&\coloneqq \frac1{\log\rho^{-1}}\iint R^\rho(y_1-y_2)\\&\qquad\qquad\quad\times\EE\left[\left(G_{\tau_2-s}(X-y_1)\sigma(u_s(y_1))\right)\left(G_{\tau_2-s}(X-y_2)\sigma(u_s(y_2))\right)^\top\right]\,\dif y_1\,\dif y_2.
    \end{align*}
    We now estimate the integral \cref{eq:vardecomp} in several steps.
    \begin{thmstepnv}
        \item \emph{The initial layer.} First we deal with the first half of the time interval $[0,\tau_1]$, whose contribution we will think of as an error. We have
            \begin{align}
                &\left\lvert \frac1{\log\rho^{-1}}\int_0^{\frac12 \tau_1}A(s)\,\dif s\right\rvert_{\Frob}  \overset{\cref{eq:momentbd}}\le\frac{C\langle a\rangle^2}{\log\rho^{-1}}\int_0^{\frac12\tau_1}\iint R^\rho(y_1-y_2)\prod_{i=1}^2 G_{\tau_2-s}(X-y_i)\,\dif y_1\,\dif y_2\,\dif s\notag\\
                &\qquad\overset{\cref{eq:HKwithkernelat0}}=\frac{C\langle a\rangle^2}{\log\rho^{-1}}\int_0^{\frac12\tau_1}\frac{\dif s}{\tau_2-s+\rho}=C\langle a\rangle^2\log_{\rho^{-1}}\frac{\tau_2+\rho}{\tau_2-\frac12\tau_1+\rho}.\label{eq:firsthalf}
            \end{align}
        \item \emph{Long-range and short-range terms.} We define, for $s\in [\frac12\tau_1,\tau_1]$,
            \begin{align}
                A_\EE(s)&\coloneqq \frac1{\log\rho^{-1}}\iint R^\rho(y_1-y_2)\left(G_{\tau_2-s}(X-y_1)\EE[\sigma(u_s(y_1))]\right)\\&\qquad\qquad\qquad\times\left(G_{\tau_2-s}(X-y_2)\EE[\sigma(u_s(y_2))]^\top\right)\,\dif y_1\,\dif y_2\notag\\
                \overset{\cref{eq:HKwithkernelat0}}&=\frac{\EE[\sigma(u_s(X))]^{\otimes 2}}{2(d-2)(\tau_2-s+\rho)\log\rho^{-1}}\overset{\cref{eq:Lrhodef}}=\frac{L_\rho(1-\log_\rho s,a)^{\otimes 2}}{2(d-2)(\tau_2-s+\rho)\log\rho^{-1}}\label{eq:AEs}
            \end{align}
            and
            \begin{equation}
                A_{\clC}(s)\coloneqq\frac1{\log\rho^{-1}}\iint R^\rho(y_1-y_2)\Cov(\sigma(u_s(y_1)),\sigma(u_s(y_2)))\prod_{i=1}^2G_{\tau_2-s}(X-y_i)\,\dif y_1\,\dif y_2 .
            \end{equation}
            Here, we use the notation
            $\Cov(u,v) \coloneqq \EE[uv^\top] - \EE[u]\EE[v]^\top$.
            From these definitions we observe that
            \begin{equation}
                A(s)=A_\EE(s)+A_\clC(s).\label{eq:Adecompose}
            \end{equation}
\item \emph{Estimate on $A_{\EE}(s)$.} The main contribution comes from
$A_{\EE}(s)$. From \cref{eq:AEs}, we estimate, for all $s\in[\frac{1}{2}\tau_{1},\tau_{1}]$,
that
\begin{align}
    \hspace{2em}&\hspace{-2em}\left\lvert A_{\EE}(s)(\tau_{2}-s+\rho)\log\rho^{-1}-\frac{L_{\rho}(1-\log_{\rho}\tau_{1},a)^{\otimes 2}}{2(d-2)}\right\rvert_\Frob\notag
              \\&\le C\lvert L_{\rho}(1-\log_{\rho}s,a)^{\otimes 2}-L_{\rho}(1-\log_{\rho}\tau_{1},a)^{\otimes 2}\rvert_\Frob\notag\\
 \overset{\cref{eq:decouplingbdd}}&{\le}C\langle a\rangle\lvert L_{\rho}(1-\log_{\rho}s,a)-L_{\rho}(1-\log_{\rho}\tau_{1},a)\rvert_\Frob\notag\\\overset{\cref{eq:Ltimereg}}&{\le}C\langle a\rangle^2\left(\sqrt{\log_{\rho^{-1}}\frac{\tau_{1}}{s}}+\frac{1}{\sqrt{\log\rho^{-1}}}\right)\le\frac{C\langle a\rangle^2}{\sqrt{\log\rho^{-1}}},
\label{eq:AELp}
\end{align}
with the last inequality by the assumption that $s\ge\frac{1}{2}\tau_{1}$.
This means that
\begin{align}
    \hspace{2em}&\hspace{-2em} \left\lvert\int_{\frac{1}{2}\tau_{1}}^{\tau_{1}}A_{\EE}(s)\,\dif s-\frac{L_{\rho}(1-\log_{\rho}\tau_{1},a)^{\otimes 2}}{2(d-2)}\log_{\rho^{-1}}\frac{\tau_{2}-\frac{1}{2}\tau_{1}+\rho}{\tau_{2}-\tau_{1}+\rho}\right\rvert_\Frob\nonumber \\
 &\le\frac{1}{\log\rho^{-1}}\int_{\frac{1}{2}\tau_{1}}^{\tau_{1}}\frac{\left\lvert A_{\EE}(s)(\tau_{2}-s+\rho)\log\rho^{-1}-\frac{L_{\rho}(1-\log_{\rho}\tau_{1},a)L_\rho(1-\log_\rho\tau_1,a)^\top}{2(d-2)}\right\rvert_\Frob}{\tau_{2}-s+\rho}\,\dif s\nonumber \\
    \overset{\cref{eq:AELp}}&{\le}\frac{C\langle a\rangle^2}{(\log\rho^{-1})^{3/2}}\int_{\frac{1}{2}\tau_{1}}^{\tau_{1}}\frac{\dif s}{\tau_{2}-s+\rho}\le\frac{C\langle a\rangle ^2}{(\log\rho^{-1})^{3/2}}\int_{\frac{1}{2}\tau_{1}}^{\tau_{1}}\frac{\dif s}{\tau_{2}-s}\notag\\&=\frac{C\langle a\rangle^2}{\sqrt{\log\rho^{-1}}}\log_{\rho^{-1}}\frac{\tau_{2}-\frac{1}{2}\tau_{1}}{\tau_{2}-\tau_{1}}.\label{eq:boundAEterm}
\end{align}
\item \emph{Estimate on $A_{\clC}(s)$.} The contribution of $A_{\clC}(s)$
is also an error term. We can use \cref{prop:covbd} to estimate,
for $s\in[\frac{1}{2}\tau_{1},\tau_{1}]$,
\begin{align}
    \hspace{0.1em}&\hspace{-0.1em} \lvert A_{\clC}(s)\rvert_{\mathrm{F}}\notag\le\frac{C\Lip(\sigma)^{2}\langle a\rangle^2}{(\log\rho^{-1})^{2}}\iint R^{\rho}(y_{1}-y_{2})\left(\log\frac{s^{1/2}}{\lvert y_{1}-y_{2}\rvert}+1+\frac{\lvert y_{1}-y_{2}\rvert}{s^{1/2}}\right)\\&\qquad\qquad\qquad\qquad\qquad\qquad\qquad\qquad\cdot\prod_{i=1}^{2}G_{\tau_{2}-s}(X-y_{i})\,\dif y_{1}\,\dif y_{2}\nonumber \\
    \overset{\cref{eq:Rrhobound}}&{\le}\frac{C\Lip(\sigma)^{2}\langle a\rangle^2}{(\log\rho^{-1})^{2}}\int\lvert y\rvert^{-2}\left(\log\frac{s^{1/2}}{\lvert y\rvert}+1+\frac{\lvert y\rvert}{s^{1/2}}\right)G_{2(\tau_{2}-s)}(y)\,\dif y\nonumber \\
 & =\frac{C\Lip(\sigma)^{2}\langle a\rangle^2}{(\tau_{2}-s)(\log\rho^{-1})^{2}}\int R^{\rho}(y)\left(\log\frac{s^{1/2}}{\lvert y\rvert(\tau_{2}-s)^{1/2}}+1+\frac{(\tau_{2}-s)^{1/2}\lvert y\rvert}{s^{1/2}}\right)G_{2}(y)\,\dif y\nonumber \\
 & \le\frac{C\Lip(\sigma)^{2}\langle a\rangle^2}{(\tau_{2}-s)(\log\rho^{-1})^{2}}\left(1+\log\frac{\tau_{1}}{\tau_{2}-\tau_{1}}+(\tau_{2}/\tau_{1})^{1/2}\right).\label{eq:applycovbd-1}
\end{align}
 This means that (folding $\Lip(\sigma)$ now into $C$)
\begin{equation}
\int_{\frac{1}{2}\tau_{1}}^{\tau_{1}}A_{\clC}(s)\,\dif s\le\frac{C\langle a\rangle}{\log\rho^{-1}}\left(1+\log\frac{\tau_{1}}{\tau_{2}-\tau_{1}}+(\tau_{2}/\tau_{1})^{1/2}\right)\log_{\rho^{-1}}\frac{\tau_{2}-\frac{1}{2}\tau_{1}}{\tau_{2}-\tau_{1}}.\label{eq:ACovbd-1}
\end{equation}
\item \emph{Combining the estimates.} Combining \cref{eq:vardecomp,eq:firsthalf,eq:Adecompose,eq:boundAEterm,eq:ACovbd-1}, and using the fact that
\[
\log_{\rho^{-1}}\frac{\tau_{2}-\frac{1}{2}\tau_{1}+\rho}{\tau_{2}-\tau_{1}+\rho}=\log_{\rho^{-1}}\left(1+\frac{1}{2}\cdot\frac{\tau_{1}}{\tau_{2}-\tau_{1}+\rho}\right),
\]
we obtain \cref{eq:VarofStep}. The bound \cref{eq:varofstep-simple} is then a simple derivation from \cref{eq:VarofStep} under the additional assumption $\tau_2\le 2\tau_1$. \qedhere
\end{thmstepnv}
\end{proof}
We also need a higher moment bound.
\begin{prop}
\label{prop:highermomentbd-1}Let $T>0$. There is an $\ell_{0}>2$
and a $C=C(\sigma,T)<\infty$ such that, for all $0\le\tau_{1}\le\tau_{2}\le T$
and all $\ell\in[2,\ell_{0}]$, we have
\begin{equation}
    \EE\lvert\mathcal{G}_{\tau_{2}-\tau_{1}}\mathcal{U}_{0,\tau_{1}}a(X)-a\rvert^{\ell}\le C\langle a\rangle^{\ell}\left(\log_{\rho^{-1}}\frac{\tau_{2}+\rho}{\tau_{2}-\tau_{1}+\rho}\right)^{\ell/2}.\label{eq:justdealwiththelog-1-1}
\end{equation}
\end{prop}

\begin{proof}
The proof is quite similar to that of \cite[Proposition 7.8]{DG22}.
As in the proof of \cref{prop:momentbd}, we note that
\begin{equation}
Z_{t}\coloneqq\mathcal{G}_{\tau_{2}-t}\mathcal{U}_{0,t}a(X)-a\label{eq:Ztdef}
\end{equation}
is a martingale in $t$. By the BDG inequality (see e.g.\ \cite[Proposition 4.4]{Kho14}),
we have a constant $C_{\ell}<\infty$ such that
\begin{equation}
\mathbb{E}\lvert Z_{t}\rvert^{\ell}\le C_{\ell}\mathbb{E}\lvert [Z]_{t}\rvert^{\ell/2}.\label{eq:BDG}
\end{equation}
Let $u_{t}=\clU_{0,t}a$. Just as in \cref{eq:UQVAdef},
we have
\begin{align*}
[Z]_{t} & =\frac{1}{\log\rho^{-1}}\int_{0}^{t}\iint R^{\rho}(y_{1}-y_{2})\prod_{i=1}^{2}\left(G_{\tau_{2}-s}(X-y_{i})\sigma(u_{s}(y_{i}))\right)\,\dif y_{1}\,\dif y_{2}\,\dif s.
\end{align*}
Therefore, we have by Jensen's inequality that 
\begin{align}
 & \lvert [Z]_{t}\rvert^{\ell/2}\le\left(\frac{1}{\log\rho^{-1}}\int_{0}^{t}\iint R^{\rho}(y_{1}-y_{2})\prod_{i=1}^{2}\left(G_{\tau_{2}-s}(x-y_{i})\lvert\sigma(u_{s}(y_{i}))\rvert\right)\,\dif y_{1}\,\dif y_{2}\,\dif s\right)^{\ell/2}\nonumber \\
 & \quad\le\frac{\left(\int_{t_{m-1}'}^{t}(G_{2(\tau_{2}-s)}*R^{\rho})(0)\,\dif s\right)^{\ell/2-1}}{\left(\log\rho^{-1}\right)^{\ell/2}}
 \nonumber \\& \quad\qquad\cdot
 \int_{0}^{t}\iint R^{\rho}(y_{1}-y_{2})\prod_{i=1}^{2}\left(G_{\tau_{2}-s}(x-y_{i})\lvert \sigma(u_{s}(y_{i}))\rvert^{\ell/2}\right)\,\dif y_{1}\,\dif y_{2}\,\dif s.\label{eq:Jensen-1}
\end{align}
We have
\begin{equation}
\int_{0}^{t}(G_{2(\tau_{2}-s)}*R^{\rho})(0)\,\dif s\overset{\cref{eq:HKwithkernelat0}}{=}\frac{1}{2(d-2)}\int_{0}^{t}\frac{1}{\tau_{2}-s+\rho}\,\dif s=\frac{\log\frac{\tau_{2}+\rho}{\tau_{2}-t+\rho}}{2(d-2)},\label{eq:firstlogterm-1}
\end{equation}
and also
\begin{align}
\mathbb{E} & \left[\int_{0}^{t}\iint R^{\rho}(y_{1}-y_{2})\prod_{i=1}^{2}\left(G_{\tau_{2}-s}(x-y_{i})\lvert \sigma(u_{s}(y_{i}))\rvert^{\ell/2}\right)\,\dif y_{1}\,\dif y_{2}\,\dif s\right]\nonumber \\
 & \le C\int_{0}^{t}\langle\vvvert u_{s}\vvvert_{\ell}\rangle^{\ell}\iint R^{\rho}(y_{1}-y_{2})\prod_{i=1}^{2}G_{\tau_{2}-s}(x-y_{i})\,\dif y_{1}\,\dif y_{2}\,\dif s\nonumber \\
 & \le C\langle a\rangle^{\ell}\log\frac{\tau_{2}+\rho}{\tau_{2}-t+\rho},\label{eq:secondlogterm-1}
\end{align}
with the second inequality by \cref{prop:momentbd,eq:HKwithkernelat0}.
Using \cref{eq:firstlogterm-1,eq:secondlogterm-1}
in \cref{eq:Jensen-1}, we get
\[
\mathbb{E}\lvert [Z]_{t}\rvert^{\ell/2}\le C\langle a\rangle^{\ell}\left(\frac{\log\frac{\tau_{2}+\rho}{\tau_{2}-t+\rho}}{\log\rho^{-1}}\right)^{\ell/2}.
\]
Taking $t=\tau_{1}$ and recalling \cref{eq:Ztdef,eq:BDG},
we obtain \cref{eq:justdealwiththelog-1-1}.
\end{proof}

\section{Analysis of the Markov chain\label{sec:analysis-markov-chain}}
\subsection{The time scales and definition of the Markov chain}\label{sec:timescalesMC}
Now we define a time discretization as in \cite{DG22}. Throughout this section, we fix $T\in [0,T_0]$ and $X\in\RR^d$. We define the small parameters $\delta_\rho,\gamma_\rho,\eta_\rho$ satisfying the conditions
\begin{align}
    (\log\rho^{-1})^{-1}&\ll \gamma_\rho\ll \delta_\rho^2 \ll \eta_\rho \ll 1,\label{eq:gammadeltacond}\\
    \delta_\rho^{-1}\rho^{\frac12 \gamma_\rho}&\ll 1.\label{eq:gammarhocond}
\end{align}
Here we write $f(\rho)\ll g(\rho)$ to mean that $\lim\limits_{\rho\downarrow 0} \frac{f(\rho)}{g(\rho)}=0$. This means in particular that
\begin{equation}
    \rho^{\delta_\rho}\ll\rho^{\gamma_\rho}\ll 1.\label{eq:rhodeltarhorhogammarholl1}
\end{equation}
Now we define
\begin{align}
    s_m &\coloneqq \rho^{m\delta_\rho};& s_m'&\coloneqq \rho^{m\delta_\rho+\gamma_\rho};\label{eq:smdef}\\
    t_m &\coloneqq T-s_m;&t_m'&\coloneqq T-s_m'.\label{eq:tmdef}
\end{align}
    It follows from \cref{eq:gammadeltacond} that
    \begin{equation}
        \label{eq:tmsll}
        t_{m+1}-t_m'\ll t_m'-t_m\ll t_m-t_{m-1}'
    \end{equation}
    for each $m$. We also define
    \begin{align}
        M_1(\rho,T)&\coloneqq \lceil \delta_\rho^{-1}\log_\rho T\rceil -1,\label{eq:M1rho}\\
        M_2(\rho)&\coloneqq \lfloor \delta_r\rho^{-1}\rfloor.\label{eq:M2rho}
    \end{align}
    
    We now define a process that solves the SPDE \cref{eq:ueqn}, but with the noise ``turned off'' on each time interval $[t_m,t_m']$. Specifically, we define
    \begin{equation}\label{eq:wic}
        w_t^{(M_1(\rho,T))} \coloneqq \clG_tu_0
    \end{equation}
    and, for $m\ge M_1(\rho,T)+1$,
    \begin{equation}\label{eq:Ymdef}
        Y_m\coloneqq \clG_{t_m'-t_m} w_{t_m}^{(m-1)}(X)
    \end{equation}
    and
    \begin{equation}\label{eq:wmdef}
    w_t^{(m)}\coloneqq\clU_{t_m',t}Y_m=\clU_{t_m',t}\clZ_X\clG_{t_m'-t_m} w_{t_m}^{(m-1)}\qquad\text{for }t\ge t_{m'}.
    \end{equation}
    In particular, this means that%
    \begin{equation}
        Y_m = \clG_{t_m'-t_m} \clU_{t'_{m-1},t_m}Y_{m-1}(X).\label{eq:Ymrecurrence}
    \end{equation}
    
    We need a uniform second moment bound, similar to \cite[Lemma~6.3]{DG22}:
    \begin{prop}\label{prop:uniformsecondmomentbound}
        For any $T>0$, we have a constant $C=C(T)<\infty$, independent of $\rho$, such that 
        \begin{equation}
            \sup_{m\in [M_1(\rho,T),M_2(\rho)]} \EE\lvert Y_m\rvert^2 \le C\langle \vvvert u_0\vvvert\rangle^2.\label{eq:Ymbound}
        \end{equation}
        and
        \begin{equation}
            \sup_{\substack{m\in [M_1(\rho,T),M_2(\rho)]\\ t\in [t_m',T]}} \EE\vvvert w^{(m)}_t\vvvert^2 \le C\langle \vvvert u_0\vvvert\rangle^2.\label{eq:wmbound}
        \end{equation}
    \end{prop}
    \begin{proof}
        We have by \cref{eq:Ymrecurrence} and \cref{prop:highermomentbd-1} with $\ell=2$ that there is a constant $C=C(\sigma,T)<\infty$, which we will allow to increase if necessary from line to line throughout this proof, such that
        \begin{align}
            \EE\lvert Y_m\rvert^2 &= \EE\lvert Y_{m-1}\rvert^2 + \EE\left[\EE\lvert \clG_{t_m'-t_m}\clU_{t'_{m-1},t_m}Y_{m-1}-Y_{m-1}\rvert^2\mid Y_{m-1}]\rvert\right]\notag\\
                                  &= \EE\lvert Y_{m-1}\rvert^2 + C\EE\langle Y_{m-1}\rangle^2\log_{\rho^{-1}}\frac{t_m'-t_{m-1}'+\rho}{t_m'-t_m+\rho}\notag\\
                                  &\le\EE\lvert Y_{m-1}\rvert^2+ C\EE\langle Y_{m-1}\rangle^2\log_{\rho^{-1}}\left(1+\frac{\rho^{\gamma_\rho}(\rho^{-\delta_\rho}-1)}{1-\rho^{\gamma_\rho}}\right)\notag\\&\le \EE \lvert Y_{m-1}\rvert^2 + C\delta_\rho\EE \langle Y_{m-1}\rangle^2.\label{eq:L2stepbound}
        \end{align}
        This means that
        \[\EE \lvert Y_m\rvert^2 \le (1+C\delta_\rho)\EE\lvert Y_{m-1}\rvert^2+C.\]
        By induction, this means that for all $m\in [M_1(\rho,T),M_2(\rho)]$, we have
        \[\EE \lvert Y_m\rvert^2 \le C\langle \vvvert u_0\vvvert\rangle^2(1+C\delta_\rho)^{m-M_1(\rho,T)}\le C\langle \vvvert u_0\vvvert\rangle^2 (1+C\delta_\rho)^{C\delta_\rho^{-1}}\le C\langle \vvvert u_0\vvvert\rangle^2.\]
        This completes the proof of \cref{eq:Ymbound}, and \cref{eq:wmbound} then follows using the definition \cref{eq:wmdef} and an application of \cref{prop:momentbd} with $\ell=2$.
    \end{proof}

    Now we will show that $w^{(m)}$ is a good approximation for $u$ at appropriate space-time points.
    \begin{prop}\label{prop:wmclosetou}
        We have a constant $C=C(\rho,T)<\infty$ such that, for all $x\in \RR^2$,
        \begin{equation}
            \sup_{\substack{m\in [M_1(\rho,T),M_2(T)]\\t\in [t_{m+1},T]}} \EE \lvert w_t^{(m)}(x)-u_t(x)\rvert^2 \le C\langle \vvvert u_0\vvvert\rangle^2\left(\lvert x-X\rvert^2\rho^{-m\delta_\rho}+o(1)\right),
            \label{eq:wmclosetou}
        \end{equation}
        where $o(1)$ denotes a quantity that goes to $0$ as $\rho\downarrow 0$.
    \end{prop}
    \begin{proof}
        We have, whenever $M_1(\rho,T)\le m\le M_2(\rho)$ and $t\in [t_{m+1},T]$, that
        \begin{align}
            \EE&\lvert w_t^{(m)}(x)-w_t^{(m-1)}(x)\rvert^2 = \EE \left\lvert \clU_{t_m',t}\clZ_X\clG_{t_m'-t_m} w^{(m-1)}_{t_m}(x)-\clU_{t_m,t}w^{(m-1)}_{t_m}(x)\right\rvert^2\notag\\
            \overset{\cref{eq:turnoffflatten}}&\le C\langle \vvvert w_{t_m}^{(m-1)}\vvvert\rangle^2\left(\log_{\rho^{-1}}\frac{t_m'-t_m+\rho}{t_{m+1}-t_m'+\rho}+\frac{1}{\log\rho^{-1}} + \frac{t-t_m'+\lvert x-X\rvert^2}{t_m'-t_m}\right).\label{eq:wminductivestart}
        \end{align}
        We have \begin{equation}\vvvert w_{t_m}^{(m-1)}\vvvert\le C\langle\vvvert u_0\vvvert\rangle\label{eq:wmmoment}\end{equation} by \cref{eq:wmbound}.
        We also have 
    \begin{equation}
            \log_{\rho^{-1}}\frac{t_m'-t_m+\rho}{t_{m+1}-t_m'+\rho}\le 
            \log_{\rho^{-1}}\left(1+\frac{1-\rho^{\gamma_\rho}}{\rho^{\gamma_\rho}-\rho^{\delta_\rho}}\right)
            \overset{\cref{eq:rhodeltarhorhogammarholl1}}\le C\gamma_\rho,\label{eq:logbd}
        \end{equation}
        as well as 
        \begin{equation}
            \frac{t-t_m'}{t_m'-t_m}\le \frac{T-t_m'}{t_m'-t_m} \le \frac{\rho^{\gamma_\rho}}{1-\rho^{\gamma_\rho}}\overset{\cref{eq:rhodeltarhorhogammarholl1}}\le C\rho^{\gamma_\rho}\label{eq:fracbd}
        \end{equation}
        (for sufficiently small $\rho$).
        Using \crefrange{eq:wmmoment}{eq:fracbd}, as well as the fact that $t_m'-t_m = \rho^{m\delta_\rho}(1-\rho^{\gamma_\rho})\ge C^{-1}\rho^{m\delta_\rho}$ in \cref{eq:wminductivestart}, we see that
        \[
            \EE\lvert w_t^{(m)}(x)-w_t^{(m-1)}(x)\rvert^2\le C\langle\vvvert u_0\vvvert\rangle^2\left(\gamma_\rho+\rho^{\gamma_\rho}+(\log \rho^{-1})^{-1}+\rho^{-m\delta_\rho}\lvert x-X\rvert^2\right).
        \]
        Inductively applying this inequality along with the triangle inequality, we obtain
        \begin{align*}
        &\left(\EE \lvert w_t^{(m)}(x)-u_t(x)\rvert^2\right)^2 \le \sum_{i= M_1(\rho,T)}^m\left(\EE\lvert w_t^{(i)}(x)-w_t^{(i-1)}(x)\rvert^2\right)^{1/2}
      \\&\qquad \le C\langle \vvvert u_0\vvvert\rangle\left(\delta_\rho^{-1}\gamma_\rho^{1/2}+\delta_\rho^{-1}\rho^{\gamma_\rho/2}+\delta_\rho^{-1}(\log\rho^{-1})^{-1/2}+\frac{\lvert x-X\rvert\rho^{-m\delta_\rho/2}}{1-\rho^{-\delta_\rho/2}}\right).
        \end{align*}
        The first three terms in brackets go to $0$ as $\rho\downarrow 0$ by \crefrange{eq:gammadeltacond}{eq:gammarhocond}, and the denominator $1-\rho^{-\delta_\rho/2}$ similarly goes to $1$ as $\rho\downarrow 0$. Hence \cref{eq:wmclosetou} is proved.
    \end{proof}
    \begin{prop}
        We have a constant $C=C(\sigma,T)<\infty$ such that
        \begin{equation}\label{eq:uapproxbychain}
            \EE\lvert Y_{M_2(\rho)} - u_T(X)\rvert^2 \le C\langle \vvvert u_0\vvvert\rangle^2 o(1),
    \end{equation}
    where $o(1)$ denotes a quantity that goes to $0$ as $\rho\downarrow 0$.
    \end{prop}
    \begin{proof}
        By \cref{prop:wmclosetou},  we see that
        \begin{equation}
            \left(\EE\lvert w_T^{(M_2(\rho))}(X)-u_T(X)\rvert^2\right)^{1/2}\le C\langle \vvvert u_0\vvvert\rangle o(1).\label{eq:applyapprox}
        \end{equation}
        We also have by \cref{prop:highermomentbd-1} (applied with $\ell=2$ and $\tau_1 = \tau_2 = T-t'_{M_2(\rho)} = \rho^{M_2(\rho)\delta_\rho+\gamma_\rho}$) that
        \[
            \EE\lvert Y_{M_2(\rho)}-w_T^{(M_2(\rho))}(X)\rvert^2 \le C\EE\langle Y_{M_2(\rho)}\rangle^2\log_{\rho^{-1}}(1+\rho^{M_2(\rho)\delta_\rho + \gamma_\rho -1})
        \]
        Recalling \cref{eq:M2rho}, we see that $M_2(\rho)\delta_\rho + \gamma_\rho-1\ge \delta_\rho+\gamma_\rho$, and hence
        \[
            \log_{\rho^{-1}}(1+\rho^{M_2(\rho)\delta_\rho+\gamma_\rho-1)}\le\log_{\rho^{-1}}(1+\rho^{\delta_\rho+\gamma_\rho})\le \frac C{\log\rho^{-1}}
        \]
        by \cref{eq:rhodeltarhorhogammarholl1}. Together with \cref{prop:uniformsecondmomentbound}, this implies that
        \begin{equation}
            \EE\lvert Y_{M_2(\rho)}-w_T^{(M_2(\rho))}(X)\rvert^2\le \frac{C\langle\vvvert u_0\vvvert\rangle}{\log\rho^{-1}}.\label{eq:endinterval}
        \end{equation}
        Combining \cref{eq:applyapprox,eq:endinterval}, we obtain \cref{eq:uapproxbychain}.
    \end{proof}

\begin{prop}\label{prop:stepvar}
    Let $T>0$. There is a constant $C=C(T,\sigma,d)<\infty$ such that
    for each $a\in\RR^\fkm$, we have 
    \begin{equation}\label{eq:varbd}
        \sup_{m\in [M_1(\rho,T),M_2(\rho)]} \left\lvert \delta_\rho^{-1}\Var(Y_m\mid Y_{m-1}=a) - \frac{L_\rho(m\delta_\rho,a)^{\otimes 2}}{2(d-2)}\right\vert \le C\langle a\rangle^2 o(1),
\end{equation}
where $o(1)$ denotes a quantity that goes to $0$ as $\rho\downarrow 0$.
\end{prop}
\begin{proof}
We recall the recurrence \cref{eq:Ymrecurrence}:
\[
    Y_m = \clG_{t_m'-t_m}\clU_{t'_{m-1},t_m}Y_{m-1}(X).
\]
We now apply \cref{prop:VarofStep} with 
\begin{equation}\label{eq:tau1}
\tau_1 = t_m - t_{m-1}' = \rho^{(m-1)\delta_\rho}(\rho^{\gamma_\rho}-\rho^{\delta_\rho})
\end{equation}
and
\begin{equation}\label{eq:tau2}
    \tau_2 = t_m'-t_{m-1}' = \rho^{(m-1)\delta_\rho+\gamma_\rho}(1-\rho^{\delta_\rho})\;.
\end{equation}
Since \[\frac{\tau_2}{\tau_1} = \frac{1-\rho^{\delta_\rho}}{1-\rho^{\delta_\rho-\gamma_\rho}}\to 1\qquad\text{as }\rho\downarrow 0,\] we can use \cref{eq:varofstep-simple}.
We can compute for all $m\in [M_1(\rho,T),M_2(\rho)]$ that
\begin{equation}
    \lim_{\rho\downarrow 0}\delta_\rho^{-1}\log_{\rho^{-1}}\left(1+\frac12\cdot\frac{\tau_1}{\tau_2-\tau_1+\rho}\right) = \lim_{\rho\downarrow 0}\delta_\rho^{-1}\log_{\rho^{-1}}\left(1+\frac12\cdot\frac{\rho^{-\delta_\rho}(\rho^{\gamma_\rho}-\rho^{\delta_\rho})}{1-\rho^{\gamma_\rho}+\rho^{1-m\delta_\rho}}\right)= 1\label{eq:timefactorconv}
\end{equation}
and
\begin{equation}
    \lim_{\rho\downarrow 0} \lvert \log_\rho\tau_1 -m\delta_\rho\rvert = 0,%
\end{equation}
and moreover these limits are uniform in $m$.
These two limits, along with the time-continuity of $L_\rho$ proved in \cref{lem:Ltimereg}, mean that the difference between the left side of \cref{eq:varbd} and $\delta_\rho$ times the left side of \cref{eq:varofstep-simple} converges to $0$ as $\rho\downarrow 0$, again uniformly in $m$.

To bound the right side of \cref{eq:varofstep-simple} given our choices \cref{eq:tau1,eq:tau2}, we compute (similarly to \cref{eq:timefactorconv})
\[
    \delta_\rho^{-1}\left(\log_{\rho^{-1}}\left(1+\frac{\tau_1}{\tau_2-\tau_1}\right)\right)^2= \delta_\rho^{-1}\left(\log_{\rho^{-1}}\left(1+\frac{\rho^{\gamma_\rho}-\rho^{\delta_\rho}}{\rho^{\delta_\rho}(1-\rho^{\gamma_\rho})}\right)\right)^2\le C\delta_\rho\to 0\quad\text{as }\rho\downarrow 0
    \]
    and also
$ (\delta_\rho\log\rho^{-1})^{-1}\to 0$ as $\rho\downarrow 0$
    by \cref{eq:gammadeltacond}.
    With these estimates on its components, the bound \cref{eq:varofstep-simple} leads to the conclusion \cref{eq:varbd}.
\end{proof}

We also need a higher moment bound, which will be a consequence of \cref{prop:highermomentbd-1}.
\begin{prop}\label{prop:highermoment-Ys}
    There is an $\ell >2$ such that, for all $a\in\RR^\fkm$, we have
    \begin{equation}\label{eq:highermoment-Ys}
        \sup_{m\in [M_1(\rho,T),M_2(\rho)]} \EE[ \lvert Y_m -Y_{m-1}\rvert^\ell\mid Y_{m-1}=a] \le C\langle a\rangle^\ell\delta_\rho^{\ell/2}.
\end{equation}
\end{prop}
\begin{proof}
    By \cref{eq:Ymrecurrence,prop:highermomentbd-1}, the left side of \cref{eq:highermoment-Ys} is bounded above by the right side of \cref{eq:justdealwiththelog-1-1} with the choices $\tau_2 = t_m'-t_{m-1}'$ and $\tau_1=t_m -t_{m-1}'$, namely
    \[
        C\langle a\rangle^\ell \left(\log_{\rho^{-1}}\frac{t_m'-t_{m-1}'+\rho}{t_m'-t_m+\rho}\right)^{\ell/2}.
    \]
    We note that
    \[
        \log_{\rho^{-1}}\frac{t_m'-t_{m-1}'+\rho}{t_m'-t_m+\rho} \le \log_{\rho^{-1}} \left(1+\frac{\rho^{-\delta_\rho + \gamma_\rho} - \rho^{\gamma_\rho}}{1-\rho^{\gamma_\rho}}\right)\le C\delta_\rho
    \]
    for an absolute constant $C$ by \cref{eq:gammadeltacond}. Thus we obtain \cref{eq:highermoment-Ys}.
\end{proof}

    \subsection{Convergence of the Markov chain to the forward-backward SDE}
    In this section we will use the general convergence criterion \cref{thm:MCtoDiffusion} to show that our Markov chain converges to a diffusion. We will furthermore identify the diffusion as a solution to our forward-backward SDE.

    \begin{proof}[Proof of \cref{thm:mainthm}]
        Let \begin{equation}J_\rho = L_\rho/\sqrt{2(d-2)}.\label{eq:Jrhodef}\end{equation}
    By \cref{prop:Lcompactness}, we can find a sequence $\rho_k\downarrow 0$ and a continuous function $J\colon [0,1]\times \RR^\fkm\to \RR^\fkm\otimes\RR^\fkn$ such that
\begin{equation}\label{eq:Jconv}J_{\rho_k}\to J\qquad\text{uniformly on compact subsets of }[0,1]\times\RR^\fkm.\end{equation}
For the moment, we fix the choice of subsequence, and all quantities may depend on it. Later, we will show that $J$ in fact does not depend on the choice of subsequence.

For $Q\in [0,1]$ and $a\in \RR^\fkm$, let $\tilde \Theta^J_{a,Q}$ solve the SDE problem
\begin{align}
    \dif \tilde\Theta^J_{a,Q}(q) &= J(1-q,\tilde\Theta^J_{a,Q}(q))\dif B(q),\qquad q\in (1-Q,1];\\
    \tilde\Theta^J_{a,Q}(1-Q)&=a.
\end{align}
We note that
\begin{equation}
    \tilde\Theta^J_{a,Q}(1)\overset{\mathrm{law}}= \Theta^J_{a,Q}(Q),\label{eq:thetatildetheta}
\end{equation}
where $\Theta^J_{a,Q}$ solves \cref{eq:Thetaproblem}.
Now fix $Q\in [0,1]$ and $a\in \RR^\fkm$ and let $T_k = \rho_k^{1-Q}$. Since we will now simultaneously consider the SPDE with varying choices of $\rho$, we will use the notation $\clU^{(\rho)}_{s,t}$ for the propagator to indicate the dependence on $\rho$. %
We want to apply \cref{thm:MCtoDiffusion} (with $J(q,\cdot)$ replaced by $J(1-q,\cdot)$) to show that, for fixed $X\in\RR^d$, we have
\begin{equation}\label{eq:SDEconv}
\clU^{(\rho_k)}_{0,T_k}a(X) %
\xrightarrow[k\to\infty]{\mathrm{law}} \tilde\Theta^J_{a,Q}(1).
\end{equation}
We consider the Markov chain $(Y_m)$ introduced in \cref{sec:timescalesMC} with $T=T_k$, $\rho = \rho_k$, and $u_0 \equiv a$. We let $A_1(k) = M_1(\rho_k,T_k)+1$ and $A_2(k) = M_2(\rho_k)$. The definitions \crefrange{eq:wic}{eq:Ymdef} mean that
\cref{eq:icconv} is satisfied with $Z\equiv a$. The \ordinalstringnum{\getrefnumber{enu:lipschitz}} condition of \cref{thm:MCtoDiffusion} is satisfied by \cref{lem:Lspacereg,eq:Jconv}. By the definitions \crefrange{eq:M1rho}{eq:M2rho}, the condition \cref{eq:endpointsconverge} is satisfied with $A_1 = 1-Q$ and $A_2 = 1$. Also, by \cref{eq:Jconv,prop:stepvar}, we see that \cref{eq:variancesconverge} is satisfied, and by \cref{prop:highermoment-Ys}, we see that \cref{eq:momentboundreq} is satisfied. Therefore, all of the conditions of \cref{thm:MCtoDiffusion} are satisfied, and hence \cref{eq:SDEconv} holds.

We can now conclude that
\begin{equs}
        J(Q,a) \overset{\cref{eq:Jconv}}&= \lim_{k\to\infty} J_{\rho_k}(Q,a) \overset{\substack{\cref{eq:Jrhodef},\\\cref{eq:Lrhodef}}}=  \frac{\lim\limits_{k\to\infty} \EE[\sigma(\clU^{(\rho_k)}_{0,T_k} a(X))]}{\sqrt{2(d-2)}} \overset{\cref{eq:SDEconv}}= \frac{\EE[\sigma(\tilde{\Theta}^J_{a,Q}(1))]}{\sqrt{2(d-2)}}\\
        \overset{\cref{eq:thetatildetheta}}&= \frac{\EE[\sigma(\Theta^J_{a,Q}(Q))]}{\sqrt{2(d-2)}} \overset{\cref{eq:Qdef}}= \clQ_{\sigma/\sqrt{2(d-2)}}J(Q,a).\label{eq:Jrelation}
\end{equs}
In the third identity we additionally used the fact that the family of random variables $(\sigma(\clU^{(\rho_k)}_{0,T_k}a(X)))_k$ is uniformly integrable by \cref{prop:momentbd} and the assumption that $\sigma$ is Lipschitz. But \cref{eq:Jrelation} means that $J$ satisfies \cref{eq:Qfp}, and hence by \cref{prop:Jcharacterization} that $J$ is uniquely determined, and therefore coincides with the function $J$ coming from the solution of \crefrange{eq:FBSDE-SDE}{eq:FBSDE-J}. In particular, it does not depend on the choice of subsequence, and so in fact \cref{eq:Jconv} can be upgraded to convergence as $\rho\downarrow 0$. 

We can now apply \cref{thm:MCtoDiffusion} again, equipped this time with this stronger notion of convergence. Let $\rho_k\downarrow 0$ be \emph{any} sequence decreasing to zero. We consider general $u_0$ satisfying the conditions of the theorem, fix $T=t$, and let $A_1(k) = M_1(\rho_k,T)$ and $A_2(k)=M_2(\rho_k)$.
We note by the definitions \crefrange{eq:wic}{eq:Ymdef} that (again using the Markov chain $(Y_m)$ from \cref{sec:timescalesMC} with these choices), we have
\[
    Y_{A_1(k)} = \clG_{t_{M_1(\rho_k,T_k))}'}u_0(x) \to \clG_t u_0(x)\qquad\text{in probability as }k\to\infty
\]
since
$t_{M_1(\rho_k,T_k))}' \to t$ as $k\to\infty$ by \cref{eq:rhodeltarhorhogammarholl1,eq:smdef,eq:tmdef,eq:M1rho}.
Thus, using \cref{thm:MCtoDiffusion} in the same way as above with $Q=0$, we get
\[
    \clU^{(\rho)}_{0,t} u_0(x)\xrightarrow[\rho\downarrow 0]{\mathrm{law}}\Gamma_{\clG_tu_0(x),1}(1),
\]
where $\Gamma_{\clG_{t}u_0(x),1}$ solves \crefrange{eq:FBSDE-SDE}{eq:FBSDE-ic} with $Q=1$.
We can then upgrade this convergence to the Wasserstein convergence \cref{eq:Wasserstein-convergence} using the moment bound in \cref{prop:momentbd}.
\end{proof}

\section{Analysis of the second moment in the scalar linear case\label{sec:analysis-second-moment}}
In this section we prove \cref{thm:nosecondmomentphasetrans}. The proof is independent of the rest of the paper, and uses different methods. Specifically, since the problem is linear in the initial data, we can write a closed PDE for the two-point correlation function of the solution. A variant of this PDE was studied in \cite{APP09}, and we use an explicit supersolution given there to bound the solution.
\begin{proof}[Proof of \cref{thm:nosecondmomentphasetrans}]
As in the theorem statement, we assume that $\fkm = \fkn = 1$ and $\sigma(u)=\beta u$ for some $\beta\in (0,\infty)$, and let $(u_t(x))$ solve \cref{eq:ueqn} with initial condition $u_0\equiv 1$. Fix $T<\infty$ arbitrarily. We define the function $k\colon\RR^d\to\RR$ by
\begin{equation}\label{eq:kdev}
    k_t(x)\coloneqq \EE [u_t(0)u_t(x)].
\end{equation}
Then, by Itô's formula applied to the SPDE \cref{eq:ueqn}, the function $k$ satisfies the PDE
\begin{subequations}
\begin{align}
    \partial_t k_t(x) &= \left(\Delta+\frac{\beta^2}{\log\rho^{-1}}R^\rho(x)\right)k_t(x).\label{eq:kpde}\\
    k_0(x) &= 1.
\end{align}
\end{subequations}
Let $v$ solve the PDE \cref{eq:kpde} but with initial condition $v_0=\delta_0$ (a delta distribution at the origin). %
By the symmetry of the right side of \cref{eq:kpde}, we have
\begin{equation}\label{eq:useselfadjoint}
    \EE[u_t(x)^2] = k_t(0) = \int_{\RR^d} v_t(x)\,\dif x.
\end{equation}

By the comparison principle, we have
\begin{equation}
    v_t(x) \le \exp\left\{\frac{\beta^2 R^\rho(0)t}{\log \rho^{-1}}\right\}G_{2t}(x).\label{eq:vteasycomp}
\end{equation}
In particular, this means that that
\begin{equation}\label{eq:vrhoestimate}
    v_{\rho/2}(x) \le \exp\left\{\frac{\beta^2 R^\rho(0)\rho}{2\log \rho^{-1}}\right\}G_{2\rho}(x)\overset{\cref{eq:HKwithkernelat0}}=\exp\left\{\frac{\beta^2}{4(d-2)\log \rho^{-1}}\right\}G_{\rho}(x)\le 2G_{\rho}(x)
\end{equation}
if $\rho$ is sufficiently small.

Let $A<\infty$  be large enough that \begin{equation}\label{eq:Acond}R^\rho(x)\le A\lvert x\rvert^{-2}\qquad\text{ for all }x\in\RR^d\setminus\{0\}\text{ and all }\rho>0.\end{equation} (That is, let $A$ be the $C$ from \cref{eq:Rrhobound}.)
Put
\begin{equation}\label{eq:alphadef}
        \alpha_\rho %
        = \frac{d-2}2\left(1-\sqrt{1-\frac{4A\beta^2}{(d-2)^2\log\rho^{-1}}}\right)
    \end{equation}
and define
\begin{equation}
    \label{eq:tildevtdef}
        \tilde v_t(x) = t^{\alpha_\rho}\lvert x\rvert^{-\alpha_\rho}G_{2t}(x).
\end{equation}
Then we have
\begin{equation}
    \partial_t \tilde v_t(x) = \left(\Delta+\frac{A\beta^2}{\lvert x\rvert^2\log\rho^{-1}}\right)\tilde v_t(x),\label{eq:tildevteqn}
\end{equation}
as computed in \cite[§2.1,~p.~902]{APP09}.
There is a constant $C$, independent of $\rho\in(0,1]$, such that $G_1(x)\le C\lvert x\rvert^{-\alpha_\rho} G_2(x)$, which means that
\[
    G_\rho(x) = \rho^{-d/2}G_1(\rho^{-1/2}x) \le C\rho^{\alpha_\rho/2-d/2}\lvert x\rvert^{-\alpha_\rho}G_2(\rho^{-1/2}x) = C\rho^{\alpha_\rho/2}\lvert x\rvert^{-\alpha_\rho}G_{2\rho}(x)
\]
for all $x\in\RR^d$.
Using this with  \cref{eq:vrhoestimate,eq:tildevtdef}, we see that
\begin{equation}
    v_{\rho/2}(x)\le 2C \rho^{-\alpha_\rho/2} \tilde v_{\rho}(x)\qquad\text{for all }x\in\RR^d.\label{eq:timerhocomp}
\end{equation}
The condition \cref{eq:Acond} implies a comparison principle between \cref{eq:kpde,eq:tildevteqn}. Using this comparison principle along with the comparison \cref{eq:timerhocomp}, we see that
\begin{equation}
    v_t(x)\le 2C\rho^{-\alpha_\rho/2}\tilde v_{t-\rho/2}(x)\qquad\text{for all }t\ge \rho/2\text{ and all }x\in\RR^d.
\end{equation}
In particular, we have
\begin{equation}
    \int v_t(x)\,\dif x \le 2C\rho^{-\alpha_\rho/2}(t-\rho/2)^{\alpha_\rho}\int \lvert x\rvert^{\alpha_\rho} G_{2t}(x)\,\dif x.\label{eq:intvbd}
\end{equation}
On the right side of \cref{eq:intvbd}, the integral and the term $(t-\rho/2)^{\alpha_\rho}$ are both bounded above, uniformly in $\rho\in (0,1]$ and in $t\in [\rho/2,T]$. Also, from \cref{eq:alphadef} and the elementary bound $\sqrt{1-x}\ge 1-x$ for $x\in [0,1]$, we notice that for sufficiently small $\rho$, we have
\[
    \alpha_\rho \le \frac{2A\beta^2}{(d-2)\log\rho^{-1}},
\]
which means that
\[
    \rho^{-\alpha_\rho/2} = \exp\left\{\frac{\alpha_\rho}2\log\rho^{-1}\right\}\le \exp\left\{\frac{A\beta^2}{d-2}\right\}
\]
is bounded above independently of $\rho$. Hence, the right side of \cref{eq:intvbd} is bounded above independently of $\rho$ and $t\in [\rho/2,T]$. Thus, by \cref{eq:useselfadjoint}, the quantity $\sup_{t\in [\rho/2,T]}\EE[u_t(0)^2]$ is bounded uniformly in $\rho$.

To complete the proof, we must show that $\sup_{t\in[0,\rho/2]}\EE [u_t(0)^2]$ is bounded uniformly in $\rho$. But this is clear from \cref{eq:useselfadjoint,eq:vteasycomp,eq:HKwithkernelat0}.
\end{proof}

\begin{appendix}

    \section{Convergence of Markov chains to diffusions}
This section is exactly analogous to \cite[Appendix A]{DG22}. We
need a multidimensional version of that result, which again is a simple
application of the results of \cite[Section 11.2]{SV06}.
\begin{thm}
\label{thm:MCtoDiffusion}Suppose that we are given a sequence $\delta_{k}\downarrow0$,
a sequence of $\RR^{\mathfrak{m}}$-valued discrete Markov martingales
$(\{Y^{(k)}_m\}_{m\in A_{1}(k),\ldots,A_{2}(k)})_{k=1}^{\infty}$,
and a continuous function $J\colon[A_{1},A_{2}]\times\RR^\fkm\to\RR^\fkm\otimes\RR^\fkn$
satisfying the following conditions:
\begin{enumerate}
\item There is an $\RR^\fkm$-valued random variable $Z$ such that
    \begin{equation}\label{eq:icconv}
        Y^{(k)}_{A_1(k)}\xrightarrow[\mathrm{law}]{k\to\infty} Z\qquad\text{as }k\to\infty.
    \end{equation}
\item \label{enu:lipschitz}For each $q\in[A_{1},A_{2}]$, the function $J(q,\cdot)\colon\RR^{\mathfrak{m}}\to\RR^\fkm\otimes\RR^\fkn$
is Lipschitz, and the Lipschitz constant is bounded above independent
of $q$.
\item We have $\delta_{k}m\in[A_{1},A_{2}]$ for all $k\ge1$ and $m=A_{1}(k),\ldots,A_{2}(k)$,
and moreover
\begin{equation}
\lim_{k\to\infty}\delta_{k}A_{1}(k)=A_{1}\qquad\text{and}\qquad\lim_{k\to\infty}\delta_{k}A_{2}(k)=A_{2}.\label{eq:endpointsconverge}
\end{equation}
\item For each $R<\infty$, we have
\begin{equation}
\lim_{k\to\infty}\sup_{\substack{\lvert x\rvert\le R\\
A_{1}(k)\le m<A_{2}(k)
}
}\left\lvert\delta_{k}^{-1}\Var[Y^{(k)}_{m+1}\mid Y^{(k)}_m=x]-J(\delta_{k}m,x)J(\delta_km,x)^\top\right\rvert=0.\label{eq:variancesconverge}
\end{equation}
\item There is an $\ell>2$ such that, for each $R<\infty$, we have
\begin{equation}
\sup_{\substack{k<\infty,\lvert x\rvert\le R,\\
A_{1}(k)\le m<A_{2}(k)
}
}\delta_{k}^{-\ell/2}\mathbb{E}[\lvert Y^{(k)}_{m+1}-Y^{(k)}_m\rvert^{\ell}\mid Y^{(k)}_m=x]<\infty.\label{eq:momentboundreq}
\end{equation}
\end{enumerate}
Let $(Y(q))_{q\in[A_{1},A_{2}]}$ solve the stochastic differential
equation
\begin{align*}
\dif Y(q) & =J(q,Y(q))\dif B(q),\qquad q\in(A_{1},A_{2});\\
Y(A_{1}) & =X,
\end{align*}
where $B(q)$ is a standard $\RR^\fkn$-valued Brownian motion. Then we have
\[
Y^{(k)}_{A_{2}(k)}\xrightarrow[k\to\infty]{\mathrm{law}}Y(A_{2}).
\]
\end{thm}

The proof follows that of \cite[Theorem A.1]{DG22}, since the needed
results from \cite{SV06} are not specific to the one-dimensional
case. We omit the details.
\end{appendix}

\printbibliography

\end{document}